\documentclass[12pt,usenames,dvipsnames]{article}
\usepackage{amssymb,amsmath,amsthm,polynom,tikz,multirow,xcolor,nccrules}
\usetikzlibrary{arrows}

\title{Hexagonal Tiling of the Plane}
\author{Ze Zhu, Erxiao Wang\thanks{Corresponding author (wang.eric@zjnu.edu.cn).  Research was supported by National Natural Science Foundation of China NSFC-RGC 12361161603 and Key Projects of Zhejiang Natural Science Foundation LZ22A010003.}, 
Zhejiang Normal University \\
Min Yan\thanks{Research was supported by NSFC-RGC Joint Research Scheme N-HKUST607/23 and Hong Kong RGC General Research Fund 16310925.}, 
Hong Kong University of Science and Technology}

\usepackage[hidelinks, hyperindex]{hyperref}

\hyperref[sec:function]{}

\newcommand{\hbullet}{\tikz \fill (0.1,0) arc (0:180:0.1);}
\newcommand*\circled[1]{\tikz[baseline=(char.base)]{
		\node[shape=circle,draw,inner sep=0.5pt] (char) {#1};}}

\newcommand{\pa}{\partial}

\newcommand{\mc}{\mathcal}
\newcommand{\bb}{\mathbb}

\newtheorem{theorem}{Theorem}
\newtheorem{lemma}[theorem]{Lemma}

\newtheorem*{theorem*}{Theorem}

\theoremstyle{definition}
\newtheorem*{definition*}{Definition}
\newtheorem*{case*}{Case}
\newtheorem*{subcase*}{Subcase}

\theoremstyle{remark}

\numberwithin{equation}{section}

\begin{document}

\maketitle

\begin{abstract}
Since the thesis of K. Reinhardt in 1918, it is well known that there are exactly three types of convex hexagons that can tile the plane. However, the proof of the fact is far from being complete. We prove this fact, under an assumption weaker than the convexity.
\end{abstract}

\section{Introduction}

The classification of the polygons that can tile the plane has a long history. If the polygon is convex, then it can only be triangle, quadrilateral, pentagon, or hexagon. Any triangle or quadrilateral (not necessarily convex) can tile the plane. More than a century ago, Reinhardt \cite{reinhardt} showed there are exactly three types of convex hexagons that can tile the plane. He also showed five types of pentagons that can tile the plane. The classification for convex pentagons is much more complicated, and is completed only recently by Mann, McLoud-Mann, von Derau \cite{mmd}, and Rao \cite{rao}. A history of the endeavour can be found in \cite{zong}.

If the polygon is not assumed to be convex, then there is no limit on the number of sides. A recent breakthrough in this regard is the discovery of the einstein, a polygon with thirteen sides that can tile the plane but only in non-periodic ways \cite{smkg}.

This paper goes back to the hexagon. After carefully examining all the proofs that classify convex hexagons that can tile the plane, we find that only some relatively simple cases were carefully written. Besides the proof by Reinhardt \cite{reinhardt} in 1918, the only other proof was by Bollobas  \cite{bollobas} in 1963. Both were far from being complete.

Reinhardt \cite{reinhardt} consists of two parts. The first part is a general analysis of the bahavior of a large part of a planar tiling, such that all the tiles have uniform upper bound for the diameters and uniform lower bound for the areas. Reinhardt then applied the analysis to tilings of the plane by congruent convex polygons. He found that the polygon must be triangle, quadrilateral, pentagon, or hexagon. Moreover, he concluded that a tiling of the plane by congruent convex hexagons has arbitrarily large part that is edge-to-edge, and all vertices have degree $3$. In Lemma \ref{basic}, under the assumption that every vertex is the meeting place of at least three tiles, which is much weaker than the convexity, we will give a direct proof of this property.

The second part is the determination of all the convex hexagons and some convex pentagons that can tile the plane. For hexagons, Reinhardt considered all eleven possible edge length combinations (distinct alphabets mean distinct edge lengths): $abcdef$, $a^2bcde$, $a^3bcd$, $a^2b^2cd$, $a^4bc$, $a^3b^2c$, $a^2b^2c^2$, $a^5b$, $a^4b^2$, $a^3b^3$, $a^6$. Then he studied three cases of the edge length arrangements (not just the combination, but also the adjacency relations) of the hexagon in a tiling:
\begin{enumerate}
\item $a,b,x,x,x,x$: If the hexagon has two adjacent edges with lengths $a,b$, such that $a$ is different from the other five edge lengths, and $b$ is different from the other five edge lengths, then the hexagon is type 1. See the first of Figure \ref{rthesis}.
\item $a,b,c,a,b,d$ or $a,b,c,a,d,b$: If the hexagon has two opposite edges of equal length $a$, and two non-adjacent edges $b$, and two edges $c,d$, such that $a,b,c,d$ are distinct, then the hexagon is either type 1 (for $a,b,c,a,b,d$), or type 2 (for $a,b,c,a,d,b$). See the second and third of Figure \ref{rthesis}.
\item $a,b,c,a,b,c$ or $a,b,c,a,c,b$: If the hexagon has two opposite edges of equal length $a$, and two pairs of non-adjacent edges $b$ and $c$, such that $a,b,c$ are distinct, then the hexagon is either type 1 (for $a,b,c,a,b,c$) or type 2 (for $a,b,c,a,c,b$). See the fourth and fifth of Figure \ref{rthesis}.
\end{enumerate}
The three types of hexagons are illustrated in Figure \ref{notation}.

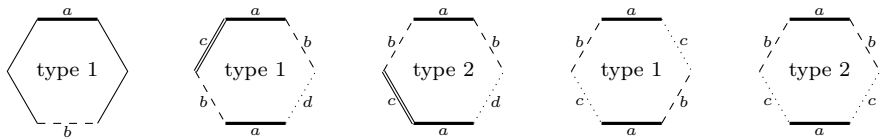
\begin{figure}[htp]
\centering
\begin{tikzpicture}[>=latex,scale=1]

\draw
	(-60:0.8) -- (0:0.8) -- (60:0.8)
	(-120:0.8) -- (180:0.8) -- (120:0.8);

\draw[line width=1.2]
	(120:0.8) -- (60:0.8);
\draw[dashed]
	(-120:0.8) -- (-60:0.8);

\node at (90:0.8) {\tiny $a$};
\node at (-90:0.8) {\tiny $b$};
\node at (0,0) {\scriptsize type 1};

\begin{scope}[xshift=2.5cm]

\draw[dashed]
	(60:0.8) -- (0:0.8)
	(180:0.8) -- (240:0.8);

\draw[line width=1.2]
	(120:0.8) -- (60:0.8)
	(-120:0.8) -- (-60:0.8);

\draw[dotted]
	(-60:0.8) -- (0:0.8);

\draw[double]
	(180:0.8) -- (120:0.8);

\node at (90:0.8) {\tiny $a$};
\node at (-90:0.8) {\tiny $a$};
\node at (30:0.8) {\tiny $b$};
\node at (150:0.8) {\tiny $c$};
\node at (-150:0.8) {\tiny $b$};
\node at (-30:0.8) {\tiny $d$};
\node at (0,0) {\scriptsize type 1};

\end{scope}

\begin{scope}[xshift=5cm]

\draw[dashed]
	(60:0.8) -- (0:0.8)
	(180:0.8) -- (120:0.8);

\draw[line width=1.2]
	(120:0.8) -- (60:0.8)
	(-120:0.8) -- (-60:0.8);

\draw[dotted]
	(-60:0.8) -- (0:0.8);

\draw[double]
	(180:0.8) -- (240:0.8);

\node at (90:0.8) {\tiny $a$};
\node at (-90:0.8) {\tiny $a$};
\node at (30:0.8) {\tiny $b$};
\node at (150:0.8) {\tiny $b$};
\node at (-150:0.8) {\tiny $c$};
\node at (-30:0.8) {\tiny $d$};
\node at (0,0) {\scriptsize type 2};

\end{scope}

\begin{scope}[xshift=7.5cm]

\draw[dotted]
	(60:0.8) -- (0:0.8)
	(180:0.8) -- (240:0.8);

\draw[line width=1.2]
	(120:0.8) -- (60:0.8)
	(-120:0.8) -- (-60:0.8);

\draw[dashed]
	(-60:0.8) -- (0:0.8)
	(180:0.8) -- (120:0.8);

\node at (90:0.8) {\tiny $a$};
\node at (-90:0.8) {\tiny $a$};
\node at (30:0.8) {\tiny $c$};
\node at (150:0.8) {\tiny $b$};
\node at (-150:0.8) {\tiny $c$};
\node at (-30:0.8) {\tiny $b$};
\node at (0,0) {\scriptsize type 1};

\end{scope}

\begin{scope}[xshift=10cm]

\draw[dashed]
	(60:0.8) -- (0:0.8)
	(180:0.8) -- (120:0.8);

\draw[line width=1.2]
	(120:0.8) -- (60:0.8)
	(-120:0.8) -- (-60:0.8);

\draw[dotted]
	(-60:0.8) -- (0:0.8)
	(180:0.8) -- (240:0.8);

\node at (90:0.8) {\tiny $a$};
\node at (-90:0.8) {\tiny $a$};
\node at (30:0.8) {\tiny $b$};
\node at (150:0.8) {\tiny $b$};
\node at (-150:0.8) {\tiny $c$};
\node at (-30:0.8) {\tiny $c$};
\node at (0,0) {\scriptsize type 2};

\end{scope}

\end{tikzpicture}
\caption{The hexagons that Reinhardt studied.}
\label{rthesis}
\end{figure}

We remark that $abcdef$, $a^2bcde$ and some subcases of the 11 cases are of the form $a,b,x,x,x,x$. More subcases are covered by the results above. However, it is quite clear that many more cases are not covered. In particular, Reinhardt did not write down any case that leads to the type 3 hexagon.

Bollobas \cite{bollobas} is independent of Reinhardt \cite{reinhardt}, because he says at the end of the paper that he did not know any hexagons other than types 1 and 2. In other words, he was not aware of type 3 hexagon. 

 Bollobas proved the square version of Lemma \ref{basic}, and the first three of the following results, and also claimed the last two results.
\begin{enumerate}
\item The hexagon cannot have six distinct edge lengths.
\item If at most two edges have equal lengths, then the hexagon is type 1.
\item The hexagon has three different corners, such that the sum of the angles of the three corners is $2\pi$. 
\item If there are only two corner combinations at all vertices in the tiling, then the hexagon is type 1 or type 2.
\item If the hexagon has two edges, such that each has distinct length from the other five, then the hexagon is type 1 or type 2. 
\end{enumerate}

Besides the works of Reinhardt and Bollobas, the only other work that contains tilings of the plane by congruent hexagons is by Heesch and Kienzle \cite{hk} in 1963. This is actually an engineering book that studied the most efficient and economical ways of cutting flat plates into identical pieces. Therefore the book was about periodic tilings by congruent polygons with edges  that are not necessarily straight. However, in the original problem of the hexagons that can tile the plane, the tiling is not assumed to be periodic.

Next, we explain three types of hexagons. 

We label the six corners of a hexagon by $i\in {\bb Z}_6$, as in the first of Figure \ref{notation}. We denote the angle at the corner $i$ by $[i]$. We also label the edge connecting corners $i$ and $i+1$ by $\bar{i}$, and denote the length of the edge by $|\bar{i}|$.

\begin{figure}[htp]
\centering
\begin{tikzpicture}[>=latex,scale=1]


\foreach \a in {0,...,5}
{
\draw[rotate=60*\a]
	(0:1.2) -- (60:1.2);

\node at (60*\a:1) {\footnotesize \a};
\node at (60*\a+30:0.85) {\footnotesize $\bar{\a}$};
}


\begin{scope}[xshift=3cm]

\draw
	(0:1.2) -- (60:1.2) -- (120:1.2)
	(180:1.2) -- (240:1.2) -- (300:1.2);

\draw[line width=1.2]
	(120:1.2) -- (180:1.2)
	(0:1.2) -- (-60:1.2);

\draw[->]
	(150:0.4) -- (150:1);
\draw[->]
	(-30:0.4) -- (-30:1);
	
\node at (0,0) {\small parallel};

\node at (0,-0.8) {$H_1$};

\end{scope}


\begin{scope}[xshift=6cm]

\draw
	(0:1.2) -- (60:1.2)
	(-60:1.2) -- (240:1.2);

\draw[line width=1.2]
	(240:1.2) -- (180:1.2)
	(60:1.2) -- (120:1.2);

\draw[dashed]
	(120:1.2) -- (180:1.2)
	(0:1.2) -- (-60:1.2);

\draw[->]
	(30:0.4) -- (0:1.1);
\draw[->]
	(150:0.4) -- (180:1.1);
\draw[->]
	(60:0.65) -- (60:1.1);
	
\node at (0,0.4) {\small $\Sigma=2\pi$};

\node at (0,-0.8) {$H_2$};

\end{scope}


\begin{scope}[xshift=9cm]

\draw[double]
	(0:1.2) -- (-60:1.2) -- (-120:1.2);

\draw[line width=1.2]
	(0:1.2) -- (60:1.2) -- (120:1.2);

\draw[dashed]
	(120:1.2) -- (180:1.2) -- (240:1.2);

\draw[->]
	(0:-0.3) -- (0:-1.1);
\draw[->]
	(120:-0.3) -- (120:-1.1);
\draw[->]
	(240:-0.3) -- (240:-1.1);
	
\node at (0,0) {\small $\frac{2}{3}\pi$};

\node at (0,-0.8) {$H_3$};

\end{scope}

\end{tikzpicture}
\caption{Labels for corners and edges, and Reinhardt hexagons.}
\label{notation}
\end{figure}
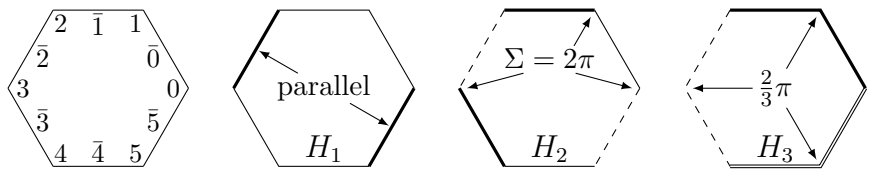

The three types of hexagons in Reinhardt \cite{reinhardt} are given by the following conditions:
\begin{itemize}
\item $H_1$: $[0]+[1]+[2]=[3]+[4]+[5]=2\pi$, $|\bar{2}|=|\bar{5}|$.
\item $H_2$: $[0]+[1]+[3]=[2]+[4]+[5]=2\pi$, $|\bar{1}|=|\bar{3}|$, $|\bar{2}|=|\bar{5}|$.
\item $H_3$: $[0]=[2]=[4]=\frac{2}{3}\pi$, $|\bar{0}|=|\bar{1}|$, $|\bar{2}|=|\bar{3}|$, $|\bar{4}|=|\bar{5}|$.
\end{itemize}
We call these {\em Reinhardt hexagons}. They are illustrated in the second, third, and fourth of Figure \ref{notation}. The thick, dashed, and double lines mean edges of the same lengths. The normal lines can have any lengths.

We will use $H_i$ to denote the type $i$ hexagon, as well as the linear system of equations satisfied by the hexagon.

\bigskip

\noindent {\bf Main Theorem.}
In a tiling of the plane by congruent hexagons, if every vertex is the meeting place of at least three tiles, then the hexagon is a Reinhardt hexagon.

\bigskip

The hexagon in the theorem is assumed to have straight edges, and is not assumed to be convex. A tiling by such hexagons has two types of vertices: {\em full} vertex $\bullet$ which is the end of all edges at the vertex, and {\em half} vertex $\hbullet$ which lies in the interior of exactly one edge (to be renamed side at the beginning of Section \ref{technical}) at the vertex. See the illustration of a quadrilateral tiling in Figure \ref{terms}. If the hexagon is convex, then the three tile assumption in the main theorem is satisfied. Therefore the theorem can be applied to convex hexagon.

\begin{figure}[htp]
\centering
\begin{tikzpicture}[>=latex,scale=1]

\draw[<->]
	(0,-0.15) -- node[fill=white] {\scriptsize side} ++(3.5,0);

\draw[<->]
	(0,0.15) -- node[fill=white] {\scriptsize edge} ++(2,0);
	
\draw[<->]
	(2,0.15) -- node[fill=white] {\scriptsize edge} ++(1.5,0);

\draw
	(0,0) -- (3.5,0) -- (4.5,1)
	(3,1) -- (2,1.5) -- (1,1) -- (0,1) -- (0,-1) -- (4.5,-1) -- (4.5,1) -- (3,1) -- (2,0) -- (1,1)
	(3.5,0) -- (3.5,-1);

\fill
	(3.5,0) circle (0.1)
	(2.1,0) arc (0:180:0.1);

\filldraw[fill=white]
	(2,1.5) circle (0.1);

\end{tikzpicture}
\label{terms}
\caption{Full vertex and half vertex.}
\end{figure}
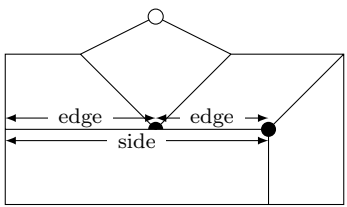

Concerning the specific cases studied by Reinhardt and Bollobas, we would like to mention the following result, which covers the first Reinhardt result and the fifth Bollobas result. We will prove the theorem in a separate paper. 

\begin{theorem*}
In an edge-to-edge tiling of the plane by congruent hexagons, suppose all vertices have degree $3$, and the hexagon has two edges, such that the length of each is different from the other five. Then the hexagon is type 1 or type 2, and the tiling is the union of the strips in Figure \ref{2edge}.
\end{theorem*}

The theorem is also valid for any part of the hexagonal tiling of the plane that satisfies the edge-to-edge and degree $3$ assumption. Lemma \ref{basic} says that, under the assumption of the main theorem, there are arbitrarily large parts of the tiling that satisfy the assumption.

\begin{figure}[htp]
\centering
\begin{tikzpicture}[>=latex,scale=1]

\pgfmathsetmacro{\ra}{(3*sqrt(3)/10)};


\foreach \a in {0,...,3}
{
\begin{scope}[xshift=2*\ra*\a cm]

\draw[xshift=\ra cm, yshift=0.9 cm]
	(-30:0.6) -- (30:0.6) -- (90:0.6) -- (150:0.6) -- (210:0.6); 

\draw
	(30:0.6) -- (-30:0.6) -- (-90:0.6) -- (210:0.6) -- (150:0.6);

\draw[line width=1.2]
	(30:0.6) -- (90:0.6);

\draw[dashed]
	(150:0.6) -- (90:0.6)
	(30:0.6) -- (30:1.2);

\foreach \b in {0,...,5}
{
\node at (60*\b+30:0.45) {\tiny \b};
\node[xshift=\ra cm, yshift=0.9 cm] at (60*\b+210:0.45) {\tiny \b};
}

\begin{scope}[yshift=-3cm]

\draw[xshift=\ra cm, yshift=0.9 cm]
	(30:0.6) -- (90:0.6) -- (150:0.6); 

\draw
	(150:0.6) -- (90:0.6) -- (30:0.6) -- (30:1.2)
	(-30:0.6) -- (-90:0.6) -- (210:0.6);
	
\end{scope}

\end{scope}
}

\begin{scope}[yshift=-3cm]

\foreach \a in {0,1}
{
\begin{scope}[xshift=4*\a*\ra cm]

\draw[line width=1.2, xshift=\ra cm, yshift=0.9 cm]
	(30:0.6) -- (-30:0.6);

\draw[dashed]
	(30:0.6) -- (-30:0.6);

\foreach \u in {0,...,5}
\node at (60*\u+150:0.45) {\tiny \u};

\foreach \u in {0,...,5}
\node[xshift=2*\ra cm] at (60*\u-30:0.45) {\tiny \u};

\foreach \u in {0,...,5}
\node[xshift=3*\ra cm, yshift=0.9 cm] at (60*\u+150:0.45) {\tiny \u};

\foreach \u in {0,...,5}
\node[xshift=\ra cm, yshift=0.9 cm] at (60*\u-30:0.45) {\tiny \u};
	
\end{scope}
}

\foreach \a in {0,1,2}
{
\begin{scope}[xshift=4*\a*\ra cm]

\draw[dashed, xshift=-\ra cm, yshift=0.9 cm]
	(30:0.6) -- (-30:0.6);

\draw[line width=1.2]
	(150:0.6) -- (210:0.6);
		
\end{scope}
}

\end{scope}


\begin{scope}[shift={(7cm,0.72 cm)}]


\foreach \x in {0,1}
{
\begin{scope}[xshift=1.8*\x cm]

\foreach \a/\b/\c in {60/0.6/1, -60/0.6/1, 180/0.6/-1, -120/1.2/-1}
{
\begin{scope}[shift={(\a:\b)}, yscale=\c, xscale=\c]

\draw
	(0:0.6) -- (60:0.6) -- (120:0.6) -- (180:0.6)
	(-60:0.6) -- (-120:0.6);

\draw[line width=1.2]
	(-60:0.6) -- (0:0.6);

\foreach \a in {0,...,5}
\node at (60*\a-120:0.45) {\tiny \a};

\end{scope}
}

\draw[dashed]
	(0,0) -- (120:0.6)
	(240:0.6) -- ++(-60:0.6);

\end{scope}
}


\begin{scope}[yshift=-3 cm]

\foreach \x in {0,1}
{
\begin{scope}[xshift=1.8*\x cm]

\foreach \a/\b/\c in {60/0.6/1, -60/0.6/1, 180/0.6/-1, -120/1.2/-1}
\draw[shift={(\a:\b)}, yscale=\c]
	(0:0.6) -- (60:0.6) -- (120:0.6) -- (180:0.6)
	(-60:0.6) -- (-120:0.6);

\draw[line width=1.2]
	(240:0.6) -- ++(-60:0.6);

\draw[dashed]
	(0,0) -- (120:0.6);	

\foreach \a in {0,...,5}
{
\node[shift={(180:0.6)}] at (60*\a+60:0.45) {\tiny \a};
\node[shift={(60:0.6)}] at (60*\a-120:0.45) {\tiny \a};
\node[shift={(-60:0.6)}] at (-60*\a-60:0.45) {\tiny \a};
\node[shift={(240:1.2)}] at (-60*\a+120:0.45) {\tiny \a};
}

\end{scope}
}

\foreach \x in {0,1,2}
{
\begin{scope}[xshift=1.8*\x cm]

\draw[line width=1.2]
	(180:1.2) -- ++(60:0.6);

\draw[shift={(210:{0.6*sqrt(3)})}, dashed]
	(0,0) -- (-120:0.6);	
	
\end{scope}
}

\end{scope}

\end{scope}

\end{tikzpicture}
\caption{Hexagonal tiling in case of two distinct edges.}
\label{2edge}
\end{figure}
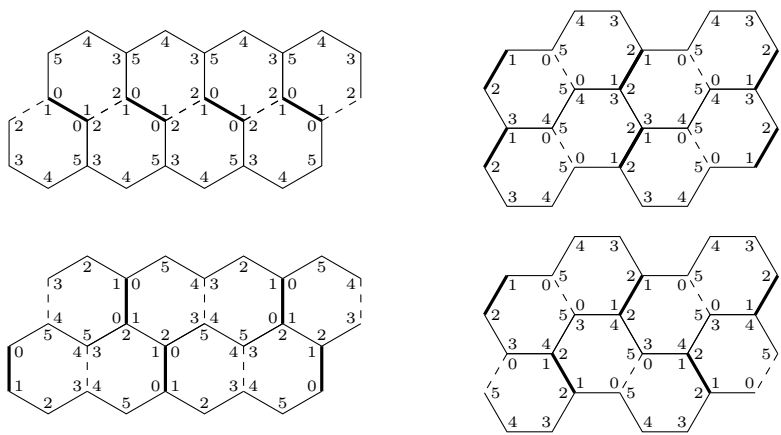

The paper is organised as follows. In Section \ref{technical}, we prove some technical results needed for the proof of the main theorem. The main consequence of Lemma \ref{basic} is that the whole hexagonal tiling of the plane contains the partial tiling in Figure \ref{geom3}, which consists of the center tile labeled 1, the first layer of tiles labeled 2 to 7, and the second layer of tiles labeled 8 to 19. We denote the center tile and the first layer by ${\mc D}_1$, and denote all the tiles in Figure \ref{geom3} by ${\mc D}_2$. The rest of the paper is based on this consequence.

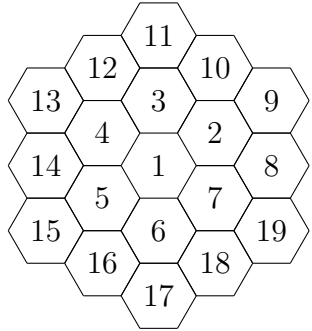
\begin{figure}[htp]
\centering
\begin{tikzpicture}[>=latex,scale=1]

\foreach \a in {0,...,5}
\foreach \b in {0,...,5}
\foreach \c/\d in {0/1.5, 30/0.866, 30/1.732}
\draw[rotate=60*\b, shift={(\c:\d)},rotate=60*\a]
	(60:0.5) -- (0:0.5);

\node at (0,0) {1};

\foreach \a in {2,...,7}
\node at (-90+60*\a:0.866) {\a};

\foreach \a in {8,10,12,14,16,18}
\node at (120+30*\a:1.5) {\a};

\foreach \a in {9,11,13,15,17,19}
\node at (120+30*\a:1.732) {\a};

\end{tikzpicture}
\caption{Two layers ${\mc D}_1,{\mc D}_2$ of tiles.}
\label{geom3}
\end{figure} 

In Section \ref{proof}, we prove the main theorem. The idea is that any ${\mc D}_1$-tiling introduces equalities: Any edge shared by adjacent tiles gives an equality between two edges of the hexagon, and the sum of three angles at any vertex is $2\pi$. If the system of equalities implies (up to circular permutation or flipping of corner labels) the system $H_i$, $i=1,2,3$, then the hexagon is one of the three types. Sometimes it is not so obvious that the system implies $H_i$, and some additional geometric argument (using some lemmas in Section \ref{technical}) or calculation can still conclude $H_i$. After dismissing all the systems that can imply Reinhardt hexagons, there remains to be seven types of hexagons, which we denote by $H_{7.1},H_{8.1},H_{8.2},H_{8.3},H_{8.4},H_{9.1},H_{9.2}$. These may not be Reinhardt hexagons, and yet can tile ${\mc D}_1$ in 33 possible ways. 

To show that the seven types of hexagons are not suitable for tiling the plane, we use the fact that the 33 possible ${\mc D}_1$-tilings can be further extended to ${\mc D}_2$-tilings. This imposes some extra geometrical conditions along the boundaries of the ${\mc D}_1$-tilings. We use the extra conditions to conclude that the hexagons either cannot tile ${\mc D}_2$, or must be reduced to the regular hexagon. 

In Section \ref{comment}, we make some comments. We discuss possible alternative ways of dismissing the remaining seven types of hexagons. This is related to a special version of the Heesch number for hexagons.

\section{Technical Preparation}
\label{technical}

The main theorem assumes a tiling of the plane by congruent hexagons, such that every vertex is the meeting place of at least three tiles. This is the overall assumption in this and subsequent sections.

Let ${\mc T}$ be a tiling satisfying the assumption. We emphasise that the hexagon may not be convex, and the tiling may not be edge-to-edge.

Let $D_R$ be a disk of radius $R>0$. Let ${\mc T}(D_R)$ be the smallest part of ${\mc T}$ that tiles a simply connected domain $\hat{D}_R$ (i.e., homeomorphic to disk), such that all tiles that touch $D_R$ are included. We may construct ${\mc T}(D_R)$ by first taking all the tiles that touch $D_R$. The union of these tiles is a topological disk minus some possible holes. These holes are the unions of some tiles. We add these tiles to form ${\mc T}(D_R)$.

Since the tiling may not be edge-to-edge, and has the interior and boundary parts, we need to fix some terminologies. First, all tiles have {\em sides}, and a hexagon has six sides. The space between adjacent sides are {\em corners}, and a hexagon has six corners. The half vertices in the interior of a side divides the side into {\em edges}. For example, the half vertex in Figure \ref{terms} divides one side of the lower left tile into two edges. 

Any tiling can be regarded as an edge-to-edge tiling, by taking all half vertices to be full vertices, and all edges to be sides. The difference is that an $n$-gonal tile $T$ becomes an $(n+h(T))$-tile. Here $h(T)$ is the number of half vertices in the interiors of the sides of $T$. For example, the quadrilateral tile in the lower left of in Figure \ref{terms} has $h(T)=1$ and becomes a pentagon in the new viewpoint.

In an edge-to-edge tiling, sides and edges are the same. The degree of a vertex is the number of edges at the vertex. If the vertex is interior, then the degree is the same as the number of corners at the vertex. If the vertex is on the boundary, then the degree is one more than the number of corners at the vertex. For example, the vertex $\circ$ in Figure \ref{terms} has one corner, and the degree is two.

If the tiling is not edge-to-edge, then the degree of a vertex is more complicated. The degree of the half vertex $\hbullet$ in Figure \ref{terms} should be $3$, the number of corners at the vertex. Fortunately, the definition is not used in this paper.

\begin{lemma}\label{basic}
Suppose a tiling of the plane by congruent hexagons has the property that every vertex is the meeting place of at least three tiles. Then for any $R>0$, there is a disk $D_R$ of radius $R$, such that the interior of ${\mc T}(D_R)$ is edge-to-edge, and all interior vertices have degree $3$.
\end{lemma}

\begin{proof}
Let ${\mc T}_E(D_R)$ be the tiling ${\mc T}(D_R)$ regarded as an edge-to-edge tiling. Let $v_R,e_R,f_R$ be the numbers of vertices, edges and tiles in ${\mc T}_E(D_R)$. Since $\hat{D}_R$ is a disk, we have the Euler equality
\begin{equation}\label{euler1}
v_R-e_R+f_R=1.
\end{equation}

Let $v_i$ be the number of degree $i$ vertices in ${\mc T}_E(D_R)$. Then $v_1=0$, and
\[
v_R=v_2+v_3+v_4+\cdots.
\]
The assumption of at least three tiles at any vertex implies that degree $2$ vertices lie in the boundary $\pa\hat{D}_R$ of $\hat{D}_R$. Moreover, we have
\begin{align}
v_2+2e_R
&=v_2+\sum_{\text{vertices $V$ of }{\mc T}_E(D_R)}\deg V \nonumber \\
&=v_2+(2v_2+3v_3+4v_4+\cdots) \nonumber \\
&=3v_R+(v_4+2v_5+3v_6+\cdots) \nonumber \\
&\ge 3v_R+v_{\ge 4}. \label{euler2}
\end{align}
Here $v_{\ge 4}=v_4+v_5+v_6+\cdots$ is the number of vertices of degree $\ge 4$.

On the other hand, we have $e_R=e_b+e_i$, where $e_b$ is the number of edges in the boundary $\pa\hat{D}_R$, and $e_i$ is the number of interior edges. Since a tile $T$ in ${\mc T}(D_R)$ becomes a $(6+h(T))$-gon in ${\mc T}_E(D_R)$, we get
\[
2e_R
\ge e_b+2e_i
=\sum_{\text{tiles $T$ of }{\mc T}(D_R)}(6+h(T)) 
=6f_R+\sum_{\text{tiles $T$ of }{\mc T}(D_R)}h(T).
\]
Since each half vertex lies in the interior of a side of exactly one $T$ in ${\mc T}(D_R)$, we know each half vertex is counted once in $h(T)$ for a unique $T$. Therefore the number $h_R$ of half vertices in the interior of ${\mc T}(D_R)$ satisfies
\[
\sum_{\text{tiles $T$ of }{\mc T}(D_R)}h(T)\ge h_R.
\]
Here the left side subtracting the right side is the number of half vertices in the boundary $\pa\hat{D}_R$. Then we get
\begin{equation}\label{euler3}
2e_R\ge 6f_R+\sum_{\text{tiles $T$ of }{\mc T}(D_R)}h(T)
\ge 6f_R+h_R.
\end{equation}

Adding \eqref{euler2} and half of \eqref{euler3} together, we get
\[
v_2+3e_R\ge 3v_R+v_{\ge 4}+3f_R+\frac{1}{2}h_R.
\]
Then by the Euler formula \eqref{euler1}, we get 
\begin{equation}\label{euler4}
v_2\ge 3+v_{\ge 4}+\frac{1}{2}h_R
>v_{\ge 4}+\frac{1}{2}h_R.
\end{equation}

Let $d$ and $a$ be the diameter and the area of the hexagon. Since degree $2$ vertices lie in the boundary $\pa\hat{D}_R$, they are the corners of tiles within the distance  $d$ of the boundary $\pa D_R$. Therefore these tiles are contained in the annulus $D_{R+d}-D_{R-d}$, and the number of such tiles is $\le\frac{\text{Area}(D_{R+d}-D_{R-d})}{a}=\frac{4\pi dR}{a}$. Since each hexagonal tile has six corners, this implies  
\[
v_2\le 6\frac{4\pi dR}{a}
=\frac{24\pi dR}{a}.
\]
Combined with \eqref{euler4}, we get the comparison of number of vertices with the area of the disk $D_R$
\[
\frac{v_{\ge 4}}{\pi R^2}+\frac{h_R}{\pi R^2}
\le \frac{v_2}{\pi R^2}
=\frac{24d}{aR}.
\]
This implies
\[
\lim_{R\to\infty}\frac{v_{\ge 4}}{\pi R^2}=0,\quad
\lim_{R\to\infty}\frac{h_R}{\pi R^2}=0.
\]
This means full vertices of degree $\ge 4$ and half vertices become more and more scarce as $R$ becomes larger. The claim of the lemma is a consequence of this fact.
\end{proof}

A consequence of Lemma \ref{basic} is that any tiling of the plane by congruent hexagons contains the ${\mc D}_2$-tiling in Figure \ref{geom3}, provided the three tile assumption in the main theorem is satisfied. In fact, the lemma also implies the tiling contains the similar ${\mc D}_k$-tiling  consisting of $k$ layers, for any $k$. Our proof of the main theorem in Section \ref{proof} only uses the ${\mc D}_2$-tiling.

\begin{lemma}\label{geom1}
If a hexagon satisfies 
\[
|\bar{0}|=|\bar{1}|,\quad
|\bar{3}|=|\bar{4}|,\quad
[0]=[2],\quad
[3]=[5],
\]
then $|\bar{2}|=|\bar{5}|$, and the hexagon is symmetric with respect to the line connecting the corners $1$ and $4$ (see the right of Figure \ref{geom1A}).
\end{lemma}

\begin{proof}
We connect corners $0$ and $2$ by a dotted line, and connect corners $3$ and $5$ by a dotted line. By $|\bar{0}|=|\bar{1}|$ and $|\bar{3}|=|\bar{4}|$, the triangles $\triangle 012$ and $\triangle 345$ are isosceles triangles. Then by $[0]=[2]$ and $[3]=[5]$, this implies equal angles $\lambda$ and equal angles $\mu$ between dotted lines and normal lines. Since the sum of  four angles $\lambda,\lambda,\mu,\mu$ of the quadrilateral $\square 0235$ is $2\pi$, we get $\lambda+\mu=\pi$. This implies the two dotted lines are parallel. Then the quadrilateral $\square 0235$ is a symmetric trapezoid, and the whole hexagon is symmetric.
\end{proof}

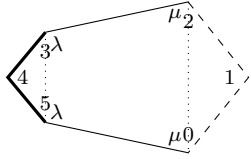
\begin{figure}[htp]
\centering
\begin{tikzpicture}[>=latex,scale=1]

\draw[line width=1.2]
	(0,-0.6) -- (-0.5,0) -- (0,0.6);

\draw
	(0,0.6) -- ++(12:1.95)
	(0,-0.6) -- ++(-12:1.95);

\draw[dashed]
	(1.9,1) -- (2.7,0) -- (1.9,-1);

\draw[dotted]
	(0,-0.6) -- (0,0.6)
	(1.9,1) -- (1.9,-1);
	
\node at (1.9,-0.75) {\scriptsize 0};
\node at (2.45,0) {\scriptsize 1};
\node at (1.9,0.75) {\scriptsize 2};
\node at (0,0.35) {\scriptsize 3};
\node at (-0.3,0) {\scriptsize 4};
\node at (0,-0.35) {\scriptsize 5};

\node at (0.15,-0.45) {\scriptsize $\lambda$};
\node at (0.15,0.45) {\scriptsize $\lambda$};

\node at (1.72,0.82) {\scriptsize $\mu$};
\node at (1.72,-0.82) {\scriptsize $\mu$};

\end{tikzpicture}
\caption{A geometric condition for symmetric hexagon.}
\label{geom1A}
\end{figure}

\begin{lemma}\label{lem52}
Suppose $|\bar{0}|=|\bar{3}|$, $|\bar{1}|=|\bar{4}|$, $|\bar{2}|=|\bar{5}|$. Then the following are equivalent:
\begin{enumerate}
\item $[0]+[2]+[4]=2\pi$.
\item $[0]=[3]$.
\item $[1]=[4]$.
\item $[2]=[5]$.
\end{enumerate}
The following are also equivalent:
\begin{enumerate}
\item $[0]+[4]+[5]=2\pi$.
\item $[2]+[5]=2\pi$.
\end{enumerate}
In both cases, $\bar{0}$ and $\bar{3}$ are parallel, and the hexagon is $H_1$. Conversely, if $\bar{0}$ and $\bar{3}$ are parallel, then the hexagon is either of the two cases.
\end{lemma}

The hexagon in the lemma is the third or the fourth in Figure \ref{lemma52fig}, with normal, thick and dashed edges indicating three length equalities. We remark that the three lengths are not assumed to be distinct. The first case means the hexagon is centrally symmetric.

In the proof, the tiles in a tiling are labeled by circled numbers, and denoted \circled{$j$}. We add subscripts to indicate which tile some corners and edges belong to. Therefore $i_j$ is the corner $i$ in \circled{$j$}.

\begin{proof}
Suppose $[0]+[2]+[4]=2\pi$. Then the three edge length equalities imply that three hexagons \circled{1}, \circled{2}, \circled{3} can be glued together, in the first of Figure \ref{lemma52fig}. Then we connect the corners together to form the dotted lines $a,b,c,d$. By the congruent triangles $\triangle_1123$ in \circled{1} and $\triangle_2123$ in \circled{2}, we get $a=b$. By the congruent triangles $\triangle_1345$ in \circled{1} and $\triangle_3345$ in \circled{3}, we get $c=d$. Therefore $a,b,c,d$ form a parallelogram. In particular, we know $c$ and $d$ are parallel. Then by the congruent triangles $\triangle_1345$ and $\triangle_3345$, this implies $\bar{0}$ and $\bar{3}$ in \circled{1} are parallel. By similar argument, we also know $\bar{1}$ and $\bar{4}$ are parallel, and $\bar{2}$ and $\bar{5}$ are parallel. This implies $[0]=[3]$, $[1]=[4]$, $[2]=[5]$. 

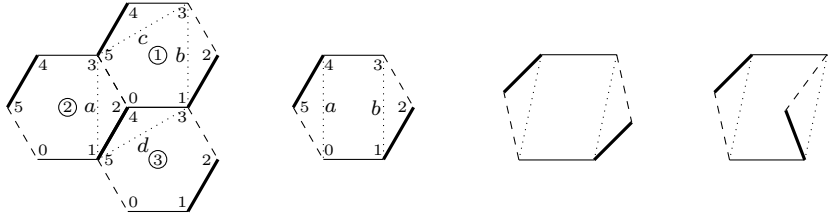
\begin{figure}[htp]
\centering
\begin{tikzpicture}[>=latex,scale=1]

\foreach \a/\b in {60/0.8, -60/0.8, 180/0.8, 0/3}
{
\begin{scope}[shift={(\a:\b)}]

\draw
	(60:0.8) -- (120:0.8)
	(-60:0.8) -- (-120:0.8);

\draw[dashed]
	(60:0.8) -- (0:0.8)
	(180:0.8) -- (240:0.8);

\draw[line width=1.2]
	(-60:0.8) -- (0:0.8)
	(180:0.8) -- (120:0.8);

\foreach \a in {0,...,5}
\node at (60*\a-120:0.65) {\tiny \a};

\end{scope}
}

\draw[dotted]
	(60:1.6) -- (0:0.8) -- (-120:0.8) -- (120:0.8) -- cycle;

\node at (-0.5,0) {\scriptsize $a$};
\node at (0.7,0.7) {\scriptsize $b$};
\node at (0.2,0.9) {\scriptsize $c$};
\node at (0.2,-0.5) {\scriptsize $d$};

\node[inner sep=0.5, draw, shape=circle] at (60:0.8) {\tiny 1};
\node[inner sep=0.5, draw, shape=circle] at (180:0.8) {\tiny 2};
\node[inner sep=0.5, draw, shape=circle] at (-60:0.8) {\tiny 3};

\begin{scope}[xshift=3cm]

\draw[dotted]
	(60:0.8) -- (-60:0.8) 
	(120:0.8) -- (-120:0.8);
	
\node at (-0.3,0) {\scriptsize $a$};
\node at (0.3,0) {\scriptsize $b$};
	
\end{scope}

\begin{scope}[shift={(8cm,-0.7cm)}]

\draw
	(0.3,1.4) -- ++(1,0)
	(0,0) -- ++(1,0);

\draw[dotted]
	(0.3,1.4) -- (0,0)
	(1.3,1.4) -- (1,0);
	
\draw[line width=1.2]
	(0.3,1.4) -- (-0.2,0.9);

\draw[dashed]
	(0,0) -- (-0.2,0.9);

\begin{scope}[xshift=1cm, rotate=-24]

\draw[line width=1.2]
	(0,0) -- (-0.5,0.5);

\draw[dashed]
	(-0.3,1.4) -- (-0.5,0.5);

\end{scope}

\end{scope}

\begin{scope}[xshift=5.2cm]

\draw
	(0.3,0.7) -- ++(1,0)
	(0,-0.7) -- ++(1,0);

\draw[dotted]
	(0.3,0.7) -- (0,-0.7)
	(1.3,0.7) -- (1,-0.7);

\draw[line width=1.2]
	(0.3,0.7) -- (-0.2,0.2)
	(1,-0.7) -- (1.5,-0.2);

\draw[dashed]
	(0,-0.7) -- (-0.2,0.2)
	(1.3,0.7) -- (1.5,-0.2);

\end{scope}

\end{tikzpicture}
\caption{Lemma \ref{lem52}: Equivalent conditions for $\bar{0}$ and $\bar{3}$ to be parallel.}
\label{lemma52fig}
\end{figure}

Suppose $[2]=[5]$ or $[2]+[5]=2\pi$. Then the edge length equalities $|\bar{1}|=|\bar{4}|$ and $|\bar{2}|=|\bar{5}|$ imply that the triangles $\triangle 123$ and $\triangle 045$ in the second of Figure \ref{lemma52fig} are congruent. This implies $a=b$. Then by $|\bar{0}|=|\bar{3}|$, we know $a,b,\bar{0},\bar{3}$ form a parallelogram $\square 0134$. This implies that $\bar{0}$ and $\bar{3}$ are parallel. The parallel property implies $[0]+[4]+[5]=2\pi$. If $[2]=[5]$, then we get $[0]+[2]+[4]=2\pi$. By what we proved earlier, this further implies $[0]=[3]$ and $[1]=[4]$. The hexagon is the third picture. If $[2]+[5]=2\pi$, then the hexagon is the fourth picture.

The argument for $[2]=[5]$ also applies to $[0]=[3]$ or $[1]=[4]$, and we get the same conclusion.

Finally, suppose $\bar{0}$ and $\bar{3}$ are parallel. Then by $|\bar{0}|=|\bar{3}|$, we know $\square 0134$ is a parallelogram. Then by $|\bar{1}|=|\bar{4}|$ and $|\bar{2}|=|\bar{5}|$, we know $\triangle 123$ and $\triangle 045$ are congruent. This implies that the hexagon is either the third or the fourth picture, and we get either of the two cases.
\end{proof}

\begin{lemma}\label{lem51}
Suppose $|\bar{0}|=|\bar{3}|$, $|\bar{1}|=|\bar{5}|$, $|\bar{2}|=|\bar{4}|$. Then the following are equivalent:
\begin{enumerate}
\item $[0]+[2]+[4]=2\pi$.
\item $[2]=[5]$.
\end{enumerate}
The following are also equivalent:
\begin{enumerate}
\item $[0]+[1]=\pi$.
\item $[2]+[5]=2\pi$.
\item $[3]+[4]=\pi$.
\end{enumerate}
In both cases, $\bar{0}$ and $\bar{3}$ are parallel, and the hexagon is $H_1$. Conversely, if $\bar{0}$ and $\bar{3}$ are parallel, then the hexagon is either of the two cases.
\end{lemma}

The hexagon in the lemma is the third or the fourth in Figure \ref{lemma51fig}, with normal, thick and dashed edges indicating three length equalities. We remark that the three lengths are not assumed to be distinct.  

\begin{proof}
Suppose $[0]+[2]+[4]=2\pi$. Then the three edge length equalities imply that three hexagons \circled{1}, \circled{2}, \circled{3} may be glued together, in the first of Figure \ref{lemma51fig}. Then we connect the corners together to form the dotted lines $a,b,c,d$. By the same argument for Lemma \ref{lem52}, we know $\bar{0}$ and $\bar{3}$ are parallel, and the parallel property implies $[0]+[4]+[5]=2\pi$. Compared with $[0]+[2]+[4]=2\pi$, we get $[2]=[5]$.  

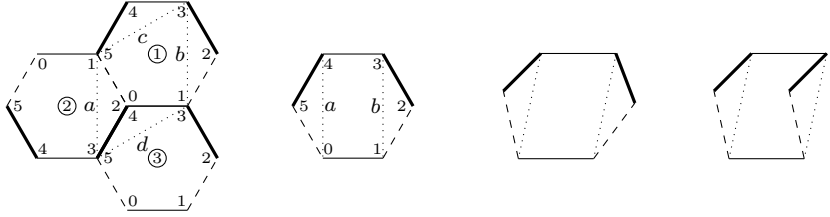
\begin{figure}[htp]
\centering
\begin{tikzpicture}[>=latex,scale=1]

\foreach \a/\b/\c in {60/0.8/1, -60/0.8/1, 180/0.8/-1, 0/3/1}
{
\begin{scope}[shift={(\a:\b)}, yscale=\c]

\draw
	(60:0.8) -- (120:0.8)
	(-60:0.8) -- (-120:0.8);

\draw[line width=1.2]
	(60:0.8) -- (0:0.8)
	(180:0.8) -- (120:0.8);

\draw[dashed]
	(-60:0.8) -- (0:0.8)
	(180:0.8) -- (240:0.8);

\foreach \a in {0,...,5}
\node at (60*\a-120:0.65) {\tiny \a};

\end{scope}
}

\draw[dotted]
	(60:1.6) -- (0:0.8) -- (-120:0.8) -- (120:0.8) -- cycle;

\node at (-0.5,0) {\scriptsize $a$};
\node at (0.7,0.7) {\scriptsize $b$};
\node at (0.2,0.9) {\scriptsize $c$};
\node at (0.2,-0.5) {\scriptsize $d$};

\node[inner sep=0.5, draw, shape=circle] at (60:0.8) {\tiny 1};
\node[inner sep=0.5, draw, shape=circle] at (180:0.8) {\tiny 2};
\node[inner sep=0.5, draw, shape=circle] at (-60:0.8) {\tiny 3};

\begin{scope}[xshift=3cm]

\draw[dotted]
	(60:0.8) -- (-60:0.8) 
	(120:0.8) -- (-120:0.8);
	
\node at (-0.3,0) {\scriptsize $a$};
\node at (0.3,0) {\scriptsize $b$};
	
\end{scope}

\begin{scope}[shift={(5.2cm,-0.7cm)}]

\draw
	(0.3,1.4) -- ++(1,0)
	(0,0) -- ++(1,0);

\draw[dotted]
	(0.3,1.4) -- (0,0)
	(1.3,1.4) -- (1,0);
	
\draw[line width=1.2]
	(0.3,1.4) -- (-0.2,0.9);

\draw[dashed]
	(0,0) -- (-0.2,0.9);

\begin{scope}[xshift=1cm, rotate=-24]

\draw[line width=1.2]
	(-0.3,1.4) -- (0.2,0.9);

\draw[dashed]
	(0,0) -- (0.2,0.9);

\end{scope}

\end{scope}

\begin{scope}[xshift=8cm]

\draw
	(0.3,0.7) -- ++(1,0)
	(0,-0.7) -- ++(1,0);

\draw[dotted]
	(0.3,0.7) -- (0,-0.7)
	(1.3,0.7) -- (1,-0.7);
	
\foreach \a in {0,1}
{
\begin{scope}[xshift=\a cm]

\draw[line width=1.2]
	(0.3,0.7) -- (-0.2,0.2);

\draw[dashed]
	(0,-0.7) -- (-0.2,0.2);

\end{scope}
}

\end{scope}

\end{tikzpicture}
\caption{Lemma \ref{lem51}: Equivalent conditions for $\bar{0}$ and $\bar{3}$ to be parallel.}
\label{lemma51fig}
\end{figure}

Suppose $[2]=[5]$ or $[2]+[5]=2\pi$. By the same argument for Lemma \ref{lem52}, applied to the second of Figure \ref{lemma51fig}, we get a parallelogram $\square 0134$. This implies that $\bar{0}$ and $\bar{3}$ are parallel, and $[0]+[4]+[5]=2\pi$. If $[2]=[5]$, then we get $[0]+[2]+[4]=2\pi$, and the hexagon is the third picture. If $[2]+[5]=2\pi$, then the parallelogram $\square 0134$ and the congruent triangles $\triangle 123$ and $\triangle 045$ imply that $\bar{1}$ and $\bar{5}$ are parallel, and  $\bar{2}$ and $\bar{4}$ are parallel. The parallel properties mean $[0]+[1]=\pi$ and $[3]+[4]=\pi$, and the hexagon is the fourth picture. 

Suppose $[0]+[1]=\pi$. Then by $|\bar{1}|=|\bar{5}|$, the quadrilateral $\square 0125$ is a parallelogram. This implies the line $l$ connecting corners $2$ and $5$ has the same length as $\bar{0}$. By $|\bar{0}|=|\bar{3}|$, we know $l$ and $\bar{3}$ have the same length. Then by $|\bar{2}|=|\bar{4}|$, we know the quadrilateral $\square 2345$ is also a parallelogram. This implies that $[3]+[4]=\pi$, and the hexagon is the fourth picture. Then we get $[2]+[5]=2\pi$.

If $[3]+[4]=\pi$, then the same argument shows $[0]+[1]=\pi$, and the hexagon is the fourth picture, and $[2]+[5]=2\pi$.

Finally, suppose $\bar{0}$ and $\bar{3}$ are parallel. Then by $|\bar{0}|=|\bar{3}|$, we know $\square 0134$ is a parallelogram. Then by $|\bar{1}|=|\bar{5}|$ and $|\bar{2}|=|\bar{4}|$, we know $\triangle 123$ and $\triangle 045$ are congruent. This implies that the hexagon is either the third or the fourth picture, and we get either of the two cases.
\end{proof}

\section{Proof of Main Theorem}
\label{proof}

By Lemma \ref{basic}, under the three tile assumption, a tiling of the plane by congruent hexagons contains the ${\mc D}_2$-tiling in Figure \ref{geom3}. The ${\mc D}_2$-tiling means assigning corners $0,1,\dots,5$ to each of the 19 tiles. Then we have an angle sum equality at each vertex (e.g., we get $[0]+[2]+[4]=2\pi$ at the vertex in the first of Figure \ref{lemma52fig}), and an edge length equality at each edge (e.g., we get $|\bar{0}|=|\bar{3}|$ from the edge shared by \circled{1}, \circled{3} in the first of Figure \ref{lemma51fig}). Moreover, we always have the following angle sum equality for hexagon 
\[
[0]+[1]+[2]+[3]+[4]+[5]=4\pi.
\]
Our task is to show that, for all ${\mc D}_2$-tilings, these equalities together imply that the hexagon is a Reinhardt hexagon.

Since the task is a huge calculation, we will first apply the idea only to the ${\mc D}_1$-tiling. We set up a computer program generating all the possible ${\mc D}_1$-tilings. Each ${\mc D}_1$-tiling ${\mc T}$ gives a collection of equalities $E({\mc T})$ as outlined above. Recall the linear systems of equations $H_1,H_2,H_3$ characterising three Reinhardt hexagons. We compare the ranks of $E({\mc T})$ and $E({\mc T})\cup \sigma(H_i)$, where $\sigma$ is any circular permutation or flipping of the corner labels $0,1,\dots,5$. If the two ranks are the same, then $E({\mc T})$ implies that the hexagon in ${\mc T}$ is $H_i$. 

We call $\text{rank}E({\mc T})$ the {\em rank} of ${\mc T}$. We find that the rank is always between $5$ and $10$. 

\subsection{Ranks 5 and 6}

Among the rank $5$ ${\mc D}_1$-tilings, we remove those ${\mc T}$ satisfying $\text{rank}E({\mc T})=5=\text{rank}(E({\mc T})\cup \sigma(H_i))$ for some $\sigma$ and $H_i$, $i=1,2,3$. What remain are two rank $5$ ${\mc D}_1$-tilings (up to circular permutation or flipping), given by the first and second of Figure \ref{rank5-6}. The respective $E({\mc T})$ is equivalent to the following:
\begin{itemize}
\item $H_{5.1}$: $|\bar{0}|=|\bar{3}|$, $|\bar{1}|=|\bar{4}|$, $|\bar{2}|=|\bar{5}|$, $[0]+[2]+[4]=[1]+[3]+[5]=2\pi$.
\item $H_{5.2}$: $|\bar{0}|=|\bar{3}|$, $|\bar{1}|=|\bar{5}|$, $|\bar{2}|=|\bar{4}|$, $[0]+[2]+[4]=[1]+[3]+[5]=2\pi$.
\end{itemize}
By Lemmas \ref{lem52} and \ref{lem51}, both imply that the hexagon in ${\mc T}$ is $H_1$. 

\begin{figure}[htp]
\centering
\begin{tikzpicture}[>=latex,scale=1]

\foreach \y in {0,...,5}
\fill[gray!50, rotate=60*\y]
	(0.4,0) -- (0.65,0) arc (0:120:0.25);
	
\foreach \c in {0,...,3}
{
\begin{scope}[xshift=3*\c cm]

\foreach \x in {0,...,5}
\draw[rotate=60*\x]
	(0:0.4) -- (60:0.4);

\foreach \x in {0,...,5}
\foreach \y in {0,...,5}
\draw[shift={(30+60*\y:0.693)}, rotate=60*\x]
	(0:0.4) -- (60:0.4);
	
\foreach \x in {0,...,5}
\node at (60*\x:0.25) {\tiny \x};

\end{scope}
}

\foreach \x in {0,...,5}
{


\foreach \y in {0,...,5}
{
\node[shift={(30+60*\y:0.693)}] at (60*\x:0.25) {\tiny \x};
\node at (30+60*\y:0.693) {\tiny $+$};
}

\node at (0,-1.5) {$5.1\colon 450123$};


\begin{scope}[xshift=3cm]

\foreach \y in {1,4}
{
\node[shift={(-30+60*\y:0.693)}] at (60*\x:0.25) {\tiny \x};
\node at (-30+60*\y:0.693) {\tiny $+$};
}

\foreach \y in {0,2,3,5}
{
\node[shift={(-30+60*\y:0.693)}] at (-60*\x-120:0.25) {\tiny \x};
\node at (-30+60*\y:0.693) {\tiny $-$};
}

\node at (0,-1.5) {$5.2\colon 4\tilde{5}\tilde{4}1\tilde{2}\tilde{1}$};

\end{scope}


\begin{scope}[xshift=6cm]

\node[shift={(30:0.693)}] at (60*\x+180:0.25) {\tiny \x};
\node[shift={(90:0.693)}] at (-60*\x-60:0.25) {\tiny \x};
\node[shift={(150:0.693)}] at (-60*\x+120:0.25) {\tiny \x};
\node[shift={(210:0.693)}] at (60*\x+180:0.25) {\tiny \x};
\node[shift={(-90:0.693)}] at (-60*\x-60:0.25) {\tiny \x};
\node[shift={(-30:0.693)}] at (-60*\x+120:0.25) {\tiny \x};

\node at (0,-1.5) {$6(1)\colon 1\tilde{0}\tilde{2}4\tilde{3}\tilde{5}$};

\end{scope}


\begin{scope}[xshift=9cm]

\node[shift={(30:0.693)}] at (60*\x+180:0.25) {\tiny \x};
\node[shift={(90:0.693)}] at (-60*\x-60:0.25) {\tiny \x};
\node[shift={(150:0.693)}] at (-60*\x+120:0.25) {\tiny \x};
\node[shift={(210:0.693)}] at (-60*\x-60:0.25) {\tiny \x};
\node[shift={(-90:0.693)}] at (60*\x+180:0.25) {\tiny \x};
\node[shift={(-30:0.693)}] at (-60*\x+120:0.25) {\tiny \x};

\node at (0,-1.5) {$6(2)\colon 1\tilde{0}\tilde{2}\tilde{4}5\tilde{5}$};

\end{scope}

}

\end{tikzpicture}
\caption{${\mc D}_1$-Tilings of rank 5 or 6.}
\label{rank5-6}
\end{figure}
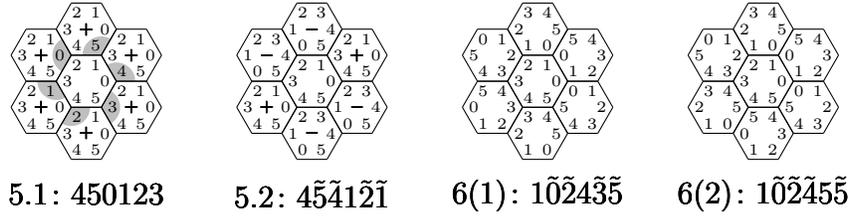

We encode ${\mc D}_1$-tilings in the following way. The center tile has an orientation given by the corners $0,1,\dots,5$. In the hexagon on the opposite side of the edge $\bar{i}$ of the center tile, we have the corner $j_i$ adjacent to the corner $i$ of the center tile. These are the six shaded corners in the first of Figure \ref{rank5-6}. Then we denote the tiling by $j_0j_1j_2j_3j_4j_5$. For example, the first tiling is denoted $450123$. If a neighborhood tile and the center tile have different orientations, then we add $\sim$ to the label. For example, the second tiling is $4\tilde{5}\tilde{4}1\tilde{2}\tilde{1}$, indicating that the tiles on the opposite sides of the edges $\bar{1},\bar{2},\bar{4},\bar{5}$ of the center tile have different orientations. 

Among rank 6 ${\mc D}_1$-tilings, we remove those ${\mc T}$ satisfying $\text{rank}E({\mc T})=6=\text{rank}(E({\mc T})\cup \sigma(H_i))$ for some $\sigma$ and $H_i$, $i=1,2,3,5.1,5.2$, because $H_{5.1}$ and $H_{5.2}$ are also Reinhardt hexagons. What remain are two rank $6$ ${\mc D}_1$-tilings, given by the third and fourth of Figure \ref{rank5-6}. They give the same $E({\mc T})$:
\begin{itemize}
\item $H_6$: $|\bar{0}|=|\bar{1}|$, $|\bar{3}|=|\bar{4}|$, $[0]=[2]$, $[3]=[5]$, $[0]+[1]+[2]=[3]+[4]+[5]=2\pi$.
\end{itemize}
By Lemma \ref{geom1}, the condition implies $|\bar{0}|=|\bar{3}|$. Moreover, by $[0]+[1]+[2]=2\pi$, we know the two edges are parallel. Therefore the hexagon is $H_1$.

\subsection{Rank 7}

Among rank 7 ${\mc D}_1$-tilings, we remove those ${\mc T}$ satisfying $\text{rank}E({\mc T})=7=\text{rank}(E({\mc T})\cup \sigma(H_i))$ for some $\sigma$ and $H_i$, $i=1,2,3,5.1,5.2,6$. What remain are four rank $7$ ${\mc D}_1$-tilings, given by Figure \ref{rank7}. They give two $E({\mc T})$:
\begin{itemize}
\item $H_{7.1}$: $|\bar{3}|=|\bar{4}|=|\bar{5}|$, $|\bar{0}|=|\bar{1}|$, $[1]=[4]=[5]=\frac{2}{3}\pi$, $[0]+[2]+[3]=2\pi$. 
\item $H_{7.2}$: $|\bar{0}|=|\bar{1}|=|\bar{3}|=|\bar{4}|$, $[0]=[2]$, $[3]=[5]$, $[0]+[2]+[4]=[1]+[3]+[5]=2\pi$.
\end{itemize}
By Lemma \ref{geom1}, $H_{7.2}$ implies $|\bar{2}|=|\bar{5}|$. Then by Lemma \ref{lem52}, we know the hexagon is $H_1$. 

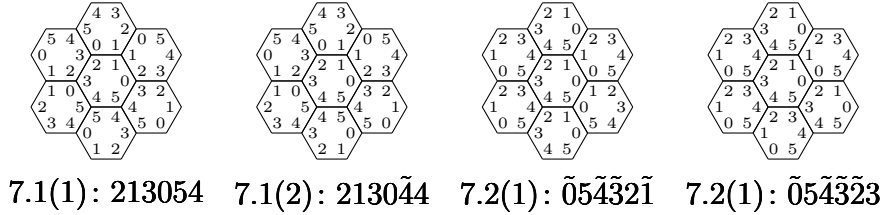
\begin{figure}[htp]
\centering
\begin{tikzpicture}[>=latex,scale=1]

\foreach \c in {0,...,3}
{
\begin{scope}[xshift=3*\c cm]

\foreach \x in {0,...,5}
\draw[rotate=60*\x]
	(0:0.4) -- (60:0.4);

\foreach \x in {0,...,5}
\foreach \y in {0,...,5}
\draw[shift={(30+60*\y:0.693)}, rotate=60*\x]
	(0:0.4) -- (60:0.4);
	
\foreach \x in {0,...,5}
\node at (60*\x:0.25) {\tiny \x};

\end{scope}
}

\foreach \x in {0,...,5}
{


\node[shift={(-90:0.693)}] at (60*\x+180:0.25) {\tiny \x};
\node[shift={(-30:0.693)}] at (60*\x-60:0.25) {\tiny \x};
\node[shift={(30:0.693)}] at (60*\x+120:0.25) {\tiny \x};
\node[shift={(90:0.693)}] at (60*\x-120:0.25) {\tiny \x};
\node[shift={(150:0.693)}] at (60*\x+180:0.25) {\tiny \x};
\node[shift={(210:0.693)}] at (60*\x+60:0.25) {\tiny \x};

\node at (0,-1.5) {$7.1(1)\colon 213054$};


\begin{scope}[xshift=3cm]

\node[shift={(-90:0.693)}] at (-60*\x:0.25) {\tiny \x};
\node[shift={(-30:0.693)}] at (60*\x+300:0.25) {\tiny \x};
\node[shift={(30:0.693)}] at (60*\x+120:0.25) {\tiny \x};
\node[shift={(90:0.693)}] at (60*\x-120:0.25) {\tiny \x};
\node[shift={(150:0.693)}] at (60*\x+180:0.25) {\tiny \x};
\node[shift={(210:0.693)}] at (60*\x+60:0.25) {\tiny \x};

\node at (0,-1.5) {$7.1(2)\colon 2130\tilde{4}4
$};

\end{scope}

\begin{scope}[xshift=6cm]

\node[shift={(-90:0.693)}] at (60*\x:0.25) {\tiny \x};
\node[shift={(-30:0.693)}] at (-60*\x+180:0.25) {\tiny \x};
\node[shift={(30:0.693)}] at (-60*\x+240:0.25) {\tiny \x};
\node[shift={(90:0.693)}] at (60*\x:0.25) {\tiny \x};
\node[shift={(150:0.693)}] at (-60*\x-120:0.25) {\tiny \x};
\node[shift={(210:0.693)}] at (-60*\x-120:0.25) {\tiny \x};

\node at (0,-1.5) {$7.2(1)\colon \tilde{0}5\tilde{4}\tilde{3}2\tilde{1}$};

\end{scope}

\begin{scope}[xshift=9cm]

\node[shift={(-90:0.693)}] at (-60*\x-120:0.25) {\tiny \x};
\node[shift={(-30:0.693)}] at (60*\x:0.25) {\tiny \x};
\node[shift={(30:0.693)}] at (-60*\x+240:0.25) {\tiny \x};
\node[shift={(90:0.693)}] at (60*\x:0.25) {\tiny \x};
\node[shift={(150:0.693)}] at (-60*\x-120:0.25) {\tiny \x};
\node[shift={(210:0.693)}] at (-60*\x-120:0.25) {\tiny \x};

\node at (0,-1.5) {$7.2(1)\colon \tilde{0}5\tilde{4}\tilde{3}\tilde{2}3$};

\end{scope}

}

\end{tikzpicture}
\caption{${\mc D}_1$-Tilings of rank 7.}
\label{rank7}
\end{figure}

For $H_{7.1}$, in addition to $[1]=[4]=[5]=\frac{2}{3}\pi$, we set
\[
|\bar{3}|=|\bar{4}|=|\bar{5}|=1,\;
|\bar{0}|=|\bar{1}|=b,\;
|\bar{2}|=a, 
\]
\[
[0]=\beta,\;
[2]=2\pi-\alpha-\beta,\;
[3]=\alpha.
\]
The hexagon is the first of Figure \ref{H7A}, with the unlabelled edges having length 1, and the lower half being half of a regular hexagon. Then the sine law of the triangles outlined by the gray lines implies
\[
a=\frac{2\cos\beta}{\sin(\frac{\pi}{6}+\alpha+\beta)},\quad
b=\frac{2\cos(\frac{\pi}{6}+\alpha)}{\sqrt{3}\sin(\frac{\pi}{6}+\alpha+\beta)}.
\]
The hexagon in Figure \ref{H7A} has specific values
\[
\alpha=100^{\circ},\quad
\beta=110^{\circ}, \quad
a=0.78986,\quad
b=0.85705,
\]
and is not a Reinhardt hexagon.

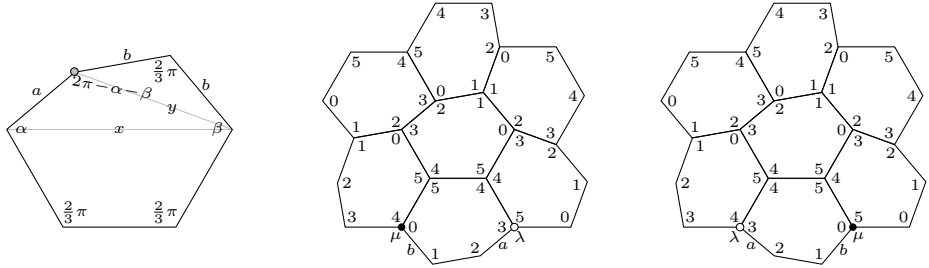
\begin{figure}[htp]
\centering
\begin{tikzpicture}[>=latex,scale=1]


\begin{scope}[yshift=0cm, scale=1]

\draw[gray!50]
	(-1.5,0) -- (1.5,0) -- ++(160:2.252);
	
\draw
	(-1.5,0) -- (240:1.5) -- (-60:1.5) -- (0:1.5) -- ++(130:1.2856) -- ++(190:1.2856) -- ++(220:1.1848);

\filldraw[fill=gray!50]
	(-0.6,0.77) circle (0.05);
		
\node at (1.15,0.6) {\tiny $b$};
\node at (0.1,1) {\tiny $b$};
\node at (-1.1,0.5) {\tiny $a$};

\node at (-0.6,-1.1) {\tiny $\frac{2}{3}\pi$};
\node at (0.6,-1.1) {\tiny $\frac{2}{3}\pi$};
\node at (0.6,0.77) {\tiny $\frac{2}{3}\pi$};
\node at (1.3,0) {\tiny $\beta$};
\node at (-1.3,0) {\tiny $\alpha$};
\node[rotate=-10] at (-0.1,0.55) {\tiny $2\pi\!-\!\alpha\!-\!\beta$};

\node at (0,0) {\tiny $x$};
\node at (0.7,0.27) {\tiny $y$};

\end{scope}

\begin{scope}[xshift=4.5cm, scale=0.5]

\foreach \a/\b/\c/\d in {1/0/0/0, 1/180/0/-2.6, 1/60/-2.25/-1.3,  1/-60/2.25/-1.3, 1/180/-2.1/0.76, 1/-120/0.15/2.06, 1/120/1.86/0.9}
{
\begin{scope}[shift={(\c cm, \d cm)}, rotate=\b, yscale=\a]

\draw
	(-1.5,0) -- (240:1.5) -- (-60:1.5) -- (0:1.5) -- ++(130:1.2856) -- ++(190:1.2856) -- ++(220:1.1848);

\node at (1.2,0) {\tiny 0};
\node at (0.6,0.7) {\tiny 1};
\node at (-0.45,0.55) {\tiny 2};
\node at (-1.17,-0.02) {\tiny 3};
\node at (-120:1.2) {\tiny 4};
\node at (-60:1.2) {\tiny 5};

\end{scope}
}

\fill (60:-3) circle (0.1);
\filldraw[fill=white] (-60:3) circle (0.1);

\node at (-1.25,-3.2) {\tiny $b$};	
\node at (1.2,-3.1) {\tiny $a$};
\node at (60:-3.3) {\tiny $\mu$};
\node at (-60:3.3) {\tiny $\lambda$};

\end{scope}

\begin{scope}[xshift=9cm, scale=0.5]

\foreach \a/\b/\c/\d in {1/0/0/0, -1/0/0/-2.6, 1/60/-2.25/-1.3,  1/-60/2.25/-1.3, 1/180/-2.1/0.76, 1/-120/0.15/2.06, 1/120/1.86/0.9}
{
\begin{scope}[shift={(\c cm, \d cm)}, rotate=\b, yscale=\a]

\draw
	(-1.5,0) -- (240:1.5) -- (-60:1.5) -- (0:1.5) -- ++(130:1.2856) -- ++(190:1.2856) -- ++(220:1.1848);
	
\node at (1.2,0) {\tiny 0};
\node at (0.6,0.7) {\tiny 1};
\node at (-0.45,0.55) {\tiny 2};
\node at (-1.17,-0.02) {\tiny 3};
\node at (-120:1.2) {\tiny 4};
\node at (-60:1.2) {\tiny 5};

\end{scope}
}

\fill (-60:3) circle (0.1);
\filldraw[fill=white] (60:-3) circle (0.1);

\node at (1.25,-3.2) {\tiny $b$};	
\node at (-1.2,-3.1) {\tiny $a$};
\node at (-60:3.3) {\tiny $\mu$};
\node at (60:-3.3) {\tiny $\lambda$};
	
\end{scope}

\end{tikzpicture}
\caption{$H_{7.1}$, and two ${\mc D}_1$-tilings.}
\label{H7A}
\end{figure}

Figure \ref{H7A} also shows two ${\mc D}_1$-tilings $213054$ and $2130\tilde{4}4$ by $H_{7.1}$. For the ${\mc D}_1$-tilings to be real, we require no overlapping among the tiles. The hexagon $H_{7.1}$ is determined by the point {\color{gray} $\bullet$} of the corner 2. We want to find the possible locations of the point, such that the ${\mc D}_1$-tilings are real.

First, a necessary condition is that the sum of each of the six angle pairs around the boundary of ${\mc D}_1$-tilings are $<2\pi$. The four pairs $[2]+[3]$, $[0]+[2]$, $[4]+[5]$, $2[1]$ are common for the two tilings, and all four are $<2\pi$. 

For the first ${\mc D}_1$-tiling $213054$, two more sums $[0]+[4]$ and $[3]+[5]$ should be $<2\pi$. This means $\alpha,\beta<\frac{4}{3}\pi$. The first of Figure \ref{H7B} shows the boundary case $[3]+[5]=2\pi$, and the second shows the boundary case $[0]+[4]=2\pi$. The boundary cases are the boundaries of the range. Then we may conclude that the range of {\color{gray} $\bullet$} for the ${\mc D}_1$-tiling to be real is the shaded infinite region, not including the boundary.

\begin{figure}[htp]
\centering
\begin{tikzpicture}[>=latex,scale=0.5]

\draw
	(1,0) -- (-60:1) -- (240:1) -- (-1,0) -- (-2.464,0) -- (-0.732,1) -- (1,0)
	(-0.732,1) -- (-0.732,3) -- ++(1,0) -- ++(-60:1) -- ++(240:1) -- ++(-30:2) -- ++(-90:2)
	(1,0) -- ++(120:1.464)
	(-60:1) -- ++(-60:1) -- ++(0:1.464) -- ++(1.464:0) -- ++(210:2)  -- ++(150:2) 
	(240:1) -- (240:2) -- ++(180:1) -- ++(240:1.464) -- ++(90:2)
	(-1,0) -- ++(210:2) -- ++(150:2) -- ++(60:1) -- ++(0:1) -- ++(-60:1)
	(-0.732,3) -- ++(240:1.464) -- ++(180:1) -- ++(240:1);

\filldraw[fill=gray!50]
	(-2.464,0) circle (0.1);

\begin{scope}[xshift=8cm]

\draw
	(-1,0) -- (240:1) -- (-60:1) -- (1,0) -- (2.4,0) -- ++(60:1.4) -- (-1,0) -- ++(60:1.4) -- ++(120:1.4) -- (-2,-1.732)
	(240:1) -- (240:2) -- ++(180:1.4) -- ++(240:1.4) -- (-60:2)
	(-60:1) -- (-60:2) -- ++(0:1) -- ++(-60:1.4) -- ++(0:1.4) -- (1,0)
	(-1,0) -- ++(60:1.4) -- ++(0:1.4) -- ++(60:1) -- ++(0:1) -- ++(60:1) -- ++(0:1) -- (3.1,-1.212) -- ++(120:1.4) -- ++(60:1.4) -- ++(120:1)
	(3.1,-1.212) -- ++(0:1) -- ++(-60:1) -- ++(240:1);

\filldraw[fill=gray!50]
	(3.1,1.212) circle (0.1);
		
\end{scope}

\begin{scope}[xshift=18cm]

\fill[gray!50]
	(-3,2) -- (-3,0) -- (-1,0) -- (240:1) -- (-60:1) -- (1,0) -- ++(30:2) -- ++(0,1);

\draw[gray!50]
	(1,0) -- (2,0)
	(1.5,0) arc (0:30:0.5);	

\node at (2,0.1) {\tiny $\frac{1}{6}\pi$};
	
\draw
	(-3,0) -- (-1,0) -- (240:1) -- (-60:1) -- (1,0) -- ++(30:2);

\end{scope}
	
\end{tikzpicture}
\caption{Range for the ${\mc D}_1$-tiling $213054$ by $H_{7.1}$.}
\label{H7B}
\end{figure}
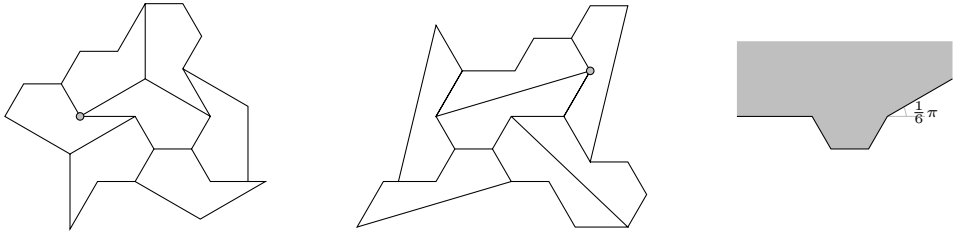

For the second ${\mc D}_1$-tiling $2130\tilde{4}4$, besides $[0]+[4]<2\pi$ and $[3]+[5]<2\pi$, which mean $\alpha,\beta<\frac{4}{3}\pi$, we also need to consider the overlapping of tiles as illustrated in Figure \ref{H7C}. In the boundary cases, the point $\bullet$ touches the tile opposite to $\bar{3}$ on the left, or touches the tile opposite to $\bar{5}$ on the right. The specific hexagon $H_{7.1}$ on the left has  
\[
\alpha=220^{\circ},\quad
\beta=99.68590^{\circ},\quad
a=1.87939,\quad
b=2.20577.
\]
The hexagon on the right has  
\[
\alpha=91.35521^{\circ},\quad
\beta=220^{\circ},\quad
a=4.79227,\quad
b=1.87939.
\]

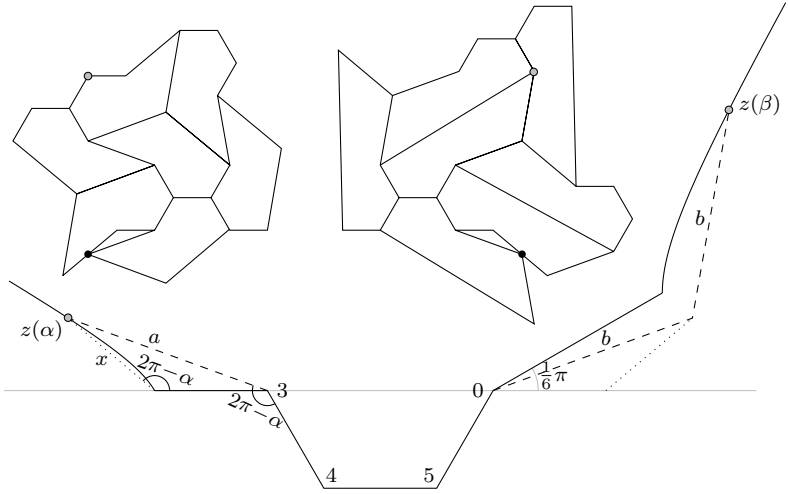
\begin{figure}[htp]
\centering
\begin{tikzpicture}[>=latex,scale=1]


\begin{scope}[shift={(-2.5cm,3cm)}, scale=0.5]

\draw
	(1,0) -- (-60:1) -- (240:1) -- (-1,0) -- ++(160:1.88) -- ++ (120:1) -- ++(60:1) -- ++(0:1) -- ++(40:1.88) -- ++(0:1) -- ++(-60:1) -- ++(-120:1) -- (1,0) -- ++(140.314:2.2058) -- ++(80.314:2.2058)
	(1,0) -- ++(140.314:2.2058) -- ++(200.314:2.2058)
	(-1,0) -- ++(200.314:2.2058) -- ++(140.314:2.2058) -- ++(60:1) -- ++(0:1)
	(240:1) -- ++(240:1) -- ++(180:1)  -- ++(220:1.88)
	(-1,0) -- ++(200.314:2.2058) -- ++(260.314:2.2058)
	(240:2) -- ++(200:1.88) -- ++(-20.314:2.2058) -- ++(40.314:2.2058)
	(-60:1) -- ++(-60:1) -- ++(0:1) -- ++(80.314:2.2058) -- ++(140.314:2.2058);

\draw[fill=gray!50]
	(-2.77,0.638) circle (0.1);
		
\fill
	(-2.77,-2.37) circle (0.1);

\begin{scope}[xshift=6cm]
	
\draw
	(-1,0) -- (240:1) -- (-60:1) -- (1,0) -- ++(20:1.88) -- ++(80:1.88) -- ++(120:1) -- ++(180:1) -- ++(240:1) -- ++(200:1.88) -- ++(260:1.88)
	(-1,0) -- ++(31.355:4.792)
	(-60:1) -- (-60:2) -- ++(0:1) -- ++(-40:1.88) -- ++(20:1.88) -- ++(60:1) -- ++(120:1) -- ++(180:1) -- ++(140:1.88)
	(1,0) -- ++(-28.645:4.792)
	(240:1) -- (240:2) -- ++(180:1) -- ++(91.355:4.792) -- ++(-40:1.88)
	(-60:2) -- ++(-20:1.88) -- ++(-80:1.88) -- (240:2)
	(1,0) -- ++(20:1.88) -- ++(80:1.88) -- ++(120:1) -- ++(60:1) -- ++(0:1) -- ++(-88.645:4.792);

\draw[fill=gray!50]
	(3.08,2.48) circle (0.1);

\fill
	(2.77,-2.37) circle (0.1);
		
\end{scope}

\end{scope}


\draw[gray!50]
	(-5,0) -- (5,0)
	(2.1,0) arc (0:30:0.6);	
	
\draw
	(-3,0) -- (-1.5,0) -- (240:1.5) -- (-60:1.5) -- (1.5,0) -- ++(30:2.598)
	(-2.8,0) arc (0:140:0.2);

\draw[scale=1.5, samples=100, smooth, domain=217:240]
	plot ({-1-(sin(\x)*cos(\x-60))/(sin(2*\x-60)}, {-(sin(\x)*sin(\x-60))/(sin(2*\x-60)})
	plot ({1-(sqrt(3)*sin(\x)*cos(-\x-90))/(sin(2*\x-60)}, {-(sqrt(3)*sin(\x)*sin(-\x-90))/(sin(2*\x-60)});
	
\draw[dotted]
	(-3,0) -- ++(140:1.505)
	(3,0) -- ++(40:1.505);

\draw[dashed]
	(-1.5,0) -- ++(160:2.82)
	(1.5,0) -- ++(20:2.82) -- ++(80:2.82) ;
	
\draw[xshift=-1.5cm, rotate=-20]
	(-0.2,0) arc (180:320:0.2);

\filldraw[fill=gray!50]
	(-4.15,0.97) circle (0.05)
	(4.635,3.735) circle (0.05);

\node at (-3,0.7) {\scriptsize $a$};
\node at (3,0.7) {\scriptsize $b$};
\node at (4.25,2.3) {\scriptsize $b$};
\node at (-3.7,0.4) {\scriptsize $x$};

\node[rotate=-20] at (-2.85,0.27) {\scriptsize $2\pi\!-\!\alpha$};
\node[rotate=-25] at (-1.65,-0.27) {\scriptsize $2\pi\!-\!\alpha$};
\node at (2.3,0.2) {\scriptsize $\frac{1}{6}\pi$};

\node at (1.3,0) {\scriptsize 0};
\node at (-1.3,0) {\scriptsize 3};
\node at (240:1.3) {\scriptsize 4};
\node at (-60:1.3) {\scriptsize 5};

\node at (-4.5,0.8) {\scriptsize $z(\alpha)$};
\node at (5.05,3.8) {\scriptsize $z(\beta)$};
			
\end{tikzpicture}
\caption{Range for the ${\mc D}_1$-tiling $2130\tilde{4}4$ by $H_{7.1}$.}
\label{H7C}
\end{figure}

The lower part of Figure \ref{H7C} shows the track of the corresponding point {\color{gray} $\bullet$}. The corners $0,3,4,5$ of the central hexagon are indicated. Four normal lines have length $1$, with the angle $\frac{2}{3}\pi$ between them, and one normal line (on the right of corner $0$) has length $\sqrt{3}$. The track of two {\color{gray} $\bullet$} are the curves $z(\alpha)$ and $z(\beta)$. All the normal lines and curves form a boundary, and the range of {\color{gray} $\bullet$} for the ${\mc D}_1$-tiling $2130\tilde{4}4$ to be real is the region above the boundary, and not including the boundary.

The location of the point {\color{gray} $\bullet$} on the left satisfies $-2+xe^{i(2\pi-\alpha)}=-1+ae^{i(\alpha-\frac{1}{3}\pi)}$. This is the same as $x=e^{i\alpha}+ae^{i(2\alpha-\frac{1}{3}\pi)}$. Taking the imaginary part, we get 
\[
a=-\frac{\sin\alpha}{\sin(2\alpha-\frac{1}{3}\pi)}, \quad
z(\alpha)=-1-\frac{\sin\alpha}{\sin(2\alpha-\frac{1}{3}\pi)}e^{i(\alpha-\frac{1}{3}\pi)}, 
\]
where $\frac{7}{6}\pi<\alpha<\frac{4}{3}\pi$, and $z(\alpha)$ is the track of the left {\color{gray} $\bullet$} point. We get the similar track for the right  {\color{gray} $\bullet$} point, with $a,\alpha$ replaced by $b,\beta$
\[
b=-\frac{\sin\beta}{\sin(2\beta-\frac{1}{3}\pi)}, \quad
z(\beta)
=1+\sqrt{3}be^{i(\frac{3}{2}\pi-\beta)}
=1-\frac{\sqrt{3}\sin\beta}{\sin(2\beta-\frac{1}{3}\pi)}e^{i(\frac{3}{2}\pi-\beta)}.
\]
Again $\frac{7}{6}\pi<\beta<\frac{4}{3}\pi$, and $z(\beta)$ is the track of the right {\color{gray} $\bullet$} point.

Next, we study how the ${\mc D}_1$-tilings can be extended to ${\mc D}_2$-tilings. This is required for the hexagon to tile the plane, and imposes additional condition on the hexagon. In the extensions, the vertices $\bullet$ and $\circ$ in the second and third of Figure \ref{H7A} have degree $3$, and are filled by third hexagons in the edge-to-edge way.

 The angles $\lambda,\mu$ at degree $3$ vertices $\bullet$ and $\circ$ are angles of the hexagon $H_{7.1}$, and satisfy
\begin{align*}
\lambda\text{-equation} &\colon
\tfrac{2}{3}\pi+\alpha+\lambda=2\pi, \\
\mu\text{-equation} &\colon
\tfrac{2}{3}\pi+\beta+\mu=2\pi.
\end{align*}
We will use the equations to argue that the hexagon $H_{7.1}$ is regular. 

If $a\ne b,1$, then the angle $\lambda$ between $a$ and $1$ equals $\alpha$. By the $\lambda$-equation, this implies $\alpha=\tfrac{2}{3}\pi$. If we further know $b\ne 1$, then we get $\mu=\beta$. If we further know $b=1$, then we get $\mu=\tfrac{2}{3}\pi,\beta$. In both cases, the $\mu$-equation implies $\beta=\tfrac{2}{3}\pi$. Then the formulae for $a$ and $b$ show $a=b=1$, contradicting the assumption $a\ne b$.

If $a=b\ne 1$, then as angles between $a=b$ and $1$, we get $\lambda,\mu=\alpha,\beta$. By the $\lambda$-equation and the $\mu$-equation, this always implies $\alpha+\beta=\tfrac{4}{3}\pi$, and $2\pi-\alpha-\beta=\tfrac{2}{3}\pi$. Therefore the upper part of $H_{7.1}$ is half of the regular hexagon of side length $a=b$. This shares the same edge $x$ with the lower part of $H_{7.1}$, which is half of the regular hexagon of side length $1$. Therefore the two regular hexagons have the same size, contradicting the assumption $a=b\ne 1$.

If $a=1\ne b$, then $\lambda=\tfrac{2}{3}\pi,\alpha$. By the $\lambda$-equation, we get $\alpha=\tfrac{2}{3}\pi$. Then by $a=1$ and $\alpha=\tfrac{2}{3}\pi$, we know $y$ is the base of the isosceles triangle with side length $1$ and top angle $\tfrac{2}{3}\pi$. Since $y$ is also the base of the isosceles triangle with side length $b$ and top angle $\tfrac{2}{3}\pi$, we get $b=1$, contradicting the assumption $a=1\ne b$.

The only remaining case is $a=b=1$. In this case, we get $\lambda,\mu=\tfrac{2}{3}\pi,\alpha,\beta,2\pi-\alpha-\beta$. Then by the $\lambda$-equation and the $\mu$-equation, we find that one of $\alpha,\beta,2\pi-\alpha-\beta$ is $\tfrac{2}{3}\pi$. 
\begin{itemize}
\item If $\alpha=\tfrac{2}{3}\pi$, then we use $y$ to divide $H_{7.1}$ into two parts. One part is the isosceles triangle of side length $1$ and top angle $\tfrac{2}{3}\pi$. The other part has four edges of length $1$ and angle $\tfrac{2}{3}\pi$ between these four edges. The union of two parts is a regular hexagon.
\item If $\beta=\tfrac{2}{3}\pi$, then $H_{7.1}$ has four consecutive angles of value $\tfrac{2}{3}\pi$. Since all edges have length $1$, this implies the hexagon is regular.
\item If $2\pi-\alpha-\beta=\tfrac{2}{3}\pi$, then we use $x$ to divide $H_{7.1}$ into two parts. Both parts  have three edges of length $1$ with angle $\tfrac{2}{3}\pi$ between them. Therefore the two parts are the two halves of the regular hexagon of side length $1$, and $H_{7.1}$ is a regular hexagon.
\end{itemize}

In later discussions, we will carry out similar argument using special features of the boundary of the ${\mc D}_1$-tiling, and conclude the hexagon is regular. In Section \ref{comment}, we will provide an alternative argument, without using the special features.  

\subsection{Rank 8}

We carry out the same process for the rank 8 ${\mc D}_1$-tilings, removing those that can imply $H_i$, $i=1,2,3,5.1,5.2,6,7.2$. What remain are seven hexagons. For each hexagon, we first list the tilings, using the notation introduced in Figures \ref{rank5-6} and \ref{rank7}. Then we list the edge length equalities, followed by the angle equalities.
\begin{itemize}
\item $H_{8.1}$: $\tilde{0}130\tilde{4}5$, $\tilde{0}130\tilde{4}\tilde{4}$, $\tilde{0}130\tilde{5}0$, $\tilde{0}13055$, $\tilde{0}1305\tilde{4}$.

$|\bar{3}|=|\bar{4}|=|\bar{5}|$, $|\bar{0}|=|\bar{1}|$,

$[0]=[1]=[4]=[5]=\frac{2}{3}\pi$, $[2]+[3]=\frac{4}{3}\pi$. 
\item $H_{8.2}$: $\tilde{0}5\tilde{2}\tilde{5}05$.

$|\bar{1}|=|\bar{3}|=|\bar{4}|=|\bar{5}|$, 

$[1]=[3]=[5]=\frac{2}{3}\pi$, $[0]=[2]$, $[0]+[2]+[4]=2\pi$. 
\item $H_{8.3}$: $\tilde{5}\tilde{0}\tilde{2}\tilde{0}\tilde{4}0$.

$|\bar{0}|=|\bar{1}|=|\bar{3}|=|\bar{5}|$, 

$[0]=[2]=[4]$, $[3]=[5]$, $[0]+[1]+[2]=[3]+[4]+[5]=2\pi$. 
\item $H_{8.4}$: 
$351X$ and $3\tilde{4}1X$, where $X$ is one of  
$555$, 
$5\tilde{4}5$, 
$55\tilde{4}$, 
$5\tilde{4}\tilde{4}$, 
$\tilde{4}55$, 
$\tilde{4}\tilde{4}5$, 
$540$, 
$5\tilde{5}\tilde{3}$, 
$\tilde{4}40$, 
$\tilde{4}\tilde{5}\tilde{3}$.

$|\bar{1}|=|\bar{3}|=|\bar{4}|=|\bar{5}|$, $|\bar{0}|=|\bar{2}|$, 

$[4]=[5]=\frac{2}{3}\pi$, $[0]+[3]=[1]+[2]=\frac{4}{3}\pi$.  
\item $H_{8.5}$: $025352$, $104431$.

$|\bar{0}|=|\bar{1}|=|\bar{5}|$, $|\bar{2}|=|\bar{3}|=|\bar{4}|$, 

$[0]=[3]$, $[1]=[4]$, $[0]+[3]+[4]=[1]+[2]+[5]=2\pi$. 
\item $H_{8.6}$: $025442$.

$|\bar{0}|=|\bar{1}|=|\bar{5}|$, $|\bar{2}|=|\bar{3}|=|\bar{4}|$, 

$[0]=[4]$, $[1]=[3]$, $[0]+[3]+[4]=[1]+[2]+[5]=2\pi$. 
\item $H_{8.7}$: $\tilde{4}\tilde{5}2\tilde{3}\tilde{0}5$.

$|\bar{0}|=|\bar{1}|=|\bar{2}|=|\bar{4}|=|\bar{5}|$, 

$[2]=[4]$, $[3]=[5]$, $[0]+[2]+[4]=[1]+[3]+[5]=2\pi$. 
\end{itemize}

The hexagons $H_{8.5}$ and $H_{8.6}$ are the first and second of Figure \ref{rank8}, with the normal edges having the same length, and the thick edges having the same length. The angle equalities imply the quadrilaterals $\square 0125$ and $\square 2345$ are similar. Since the quadrilaterals share the same dotted edge connecting corners 2 and 5, they are actually congruent. Therefore the thick and normal edges have equal length, and the hexagon is equilateral.

In $H_{8.5}$, by the congruence of the quadrilaterals, we get $[2]=[5]$. Combined with $[0]=[3]$ and $[1]=[4]$, and the angle sum for the hexagon, we get $[0]+[1]+[2]=2\pi$. Since the hexagon is equilateral, this implies that the hexagon is $H_1$ (actually centrally symmetric).

In $H_{8.6}$, the equilateral property and $[1]+[2]+[5]=2\pi$ imply that the hexagon is $H_2$. 

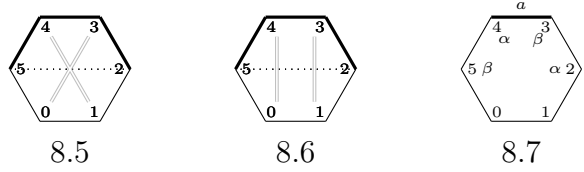
\begin{figure}[htp]
\centering
\begin{tikzpicture}[>=latex,scale=1]

\foreach \a in {0,1}
\foreach \x in {0,1,2}
{
\begin{scope}[xshift=3*\a cm]

\draw[rotate=-60*\x]
	(0:0.8) -- (-60:0.8);

\draw[rotate=60*\x, line width=1.2]
	(0:0.8) -- (60:0.8);

\draw[dotted]
	(0:0.8) -- (180:0.8);

\foreach \y in {0,...,5}
\node at (60*\y-120:0.65) {\tiny \y};

\end{scope}
}

\draw[gray!50, double]
	(60:0.5) -- (240:0.5)
	(-60:0.5) -- (120:0.5);
		
\node at (0,-1.1) {8.5};

\begin{scope}[xshift=3cm]

\draw[gray!50, double]
	(60:0.5) -- (-60:0.5)
	(240:0.5) -- (120:0.5);
		
\node at (0,-1.1) {8.6};
	
\end{scope}

\begin{scope}[xshift=6cm]

\foreach \x in {0,...,4}
\draw[rotate=60*\x]
	(120:0.8) -- (180:0.8);

\draw[line width=1.2]
	(120:0.8) -- (60:0.8);

\foreach \y in {0,...,5}
\node at (60*\y-120:0.65) {\tiny \y};

\node at (0,0.85) {\tiny $a$};
\node at (0:0.45) {\tiny $\alpha$};
\node at (120:0.45) {\tiny $\alpha$};
\node at (60:0.45) {\tiny $\beta$};
\node at (180:0.45) {\tiny $\beta$};
		
\node at (0,-1.1) {8.7};
	
\end{scope}

\end{tikzpicture}
\caption{$H_{8.5},H_{8.6},H_{8.7}$ are Reinhardt hexagons.}
\label{rank8}
\end{figure}

For the other rank 8 ${\mc D}_1$-tilings, we may actually calculate the hexagon. The calculation is based on the equation
\begin{equation}\label{hexeq}
\sum_{k=0}^5|\bar{k}|e^{i(k\pi-[1]-\cdots-[k])}=0.
\end{equation}
The idea is that a general hexagon is determined by five edge lengths and four angles between the five edges. Therefore, up to the scaling of the hexagon, there are $5+4-1=8$ free choices of variables. In principle, given a rank 8 system of linear equations among the angles and edge lengths, we may hope to determine the hexagon. In the actual calculation, only $H_{8.7}$ can be determined to be the regular hexagon. The other four hexagons $H_{8.1}$, $H_{8.2}$, $H_{8.3}$, $H_{8.4}$ allow one free parameter.

\medskip

\noindent{\em Hexagon $H_{8.7}$}

\medskip

In $H_{8.7}$, we assume $|\bar{0}|=|\bar{1}|=|\bar{2}|=|\bar{4}|=|\bar{5}|=1$, and denote (see the third of Figure \ref{rank8})
\[
|\bar{3}|=a,\;
[0]=2\pi-2\alpha,\;
[1]=2\pi-2\beta,\;
[2]=[4]=\alpha,\;
[3]=[5]=\beta.
\] 
Then \eqref{hexeq} becomes ($|\bar{0}|$ in the equality corresponds to the edge $a=|\bar{3}|$)
\[
a+e^{i(\pi-\alpha)}
+e^{i(2\pi-\alpha-\beta)}
+e^{i(\pi+\alpha-\beta)}
+e^{i(\alpha+\beta)}
+e^{i(\pi+\beta)}=0.
\]
Since $a$ and $e^{i(2\pi-\alpha-\beta)}+e^{i(\alpha+\beta)}=e^{-i(\alpha+\beta)}
+e^{i(\alpha+\beta)}$ are real, the sum of the imaginary parts of the remaining terms vanish
\[
0=-\sin\alpha+\sin\beta+\sin(\alpha-\beta)
=4\sin\tfrac{\alpha-\beta}{2}\sin\tfrac{\alpha}{2}\sin\tfrac{\beta}{2}.
\]
Since $0<\alpha,\beta<2\pi$, the equality implies $\alpha=\beta$. Therefore we get $[2]=[3]=[4]=[5]$, which means the hexagon is symmetric, and $[0]=[1]$. Then by $[0]+[2]+[5]=[0]+[2]+[4]=2\pi$, the hexagon is $H_2$.

Although we can further show that $H_{8.7}$ is regular, this is not needed for the main theorem. 

\medskip

\noindent{\em Hexagon $H_{8.1}$}

\medskip

In addition to $[0]=[1]=[4]=[5]=\frac{2}{3}\pi$, we set 
\[
|\bar{3}|=|\bar{4}|=|\bar{5}|=1,\;
|\bar{0}|=|\bar{1}|=b,\;
|\bar{2}|=a,\;
[2]=\tfrac{4}{3}\pi-\alpha,\;
[3]=\alpha.
\]
The hexagon is the first of Figure \ref{H8A} (unlabelled edges have length 1), and is the special case $\beta=\frac{2}{3}\pi$ of $H_{7.1}$
\[
a=-\frac{1}{\sin(\frac{\pi}{6}-\alpha)},\quad
b=\frac{2\cos(\frac{\pi}{6}+\alpha)}{\sqrt{3}\sin(\frac{\pi}{6}-\alpha)}.
\]
The specific hexagon $H_{8.1}$ in Figure \ref{H8A} has  
\[
\alpha=100^{\circ},\quad
a=1.06418,\quad
b=0.78986,
\]
and is not a Reinhardt hexagon.

\begin{figure}[htp]
\centering
\begin{tikzpicture}[>=latex,scale=1]


\draw[gray!50]
	(-1.5,0) -- (1.5,0) -- ++(150:1.984);

\draw
	(-1.5,0) -- (240:1.5) -- (-60:1.5) -- (0:1.5) -- ++(120:1.1453) -- ++(-1.1453,0) -- (-1.5,0);

\node at (1.3,0.6) {\tiny $b$};
\node at (0.35,1.1) {\tiny $b$};
\node at (-0.9,0.6) {\tiny $a$};

\node at (0.05,0.8) {\tiny $\frac{4}{3}\pi\!-\!\alpha$};
\node at (-1.25,0) {\tiny $\alpha$};
\node at (-0.6,-1.1) {\tiny $\frac{2}{3}\pi$};
\node at (0.6,-1.1) {\tiny $\frac{2}{3}\pi$};
\node at (1.25,0) {\tiny $\frac{2}{3}\pi$};
\node at (0.8,0.8) {\tiny $\frac{2}{3}\pi$};

\begin{scope}[xshift=4.5cm]

\end{scope}

\foreach \x/\y in {4.5/0, 9/0, 0/-4, 4.5/-4, 9/-4}
\foreach \a/\b/\c/\d in {1/0/0/0, 1/180/-1.72/0.99, 1/60/-2.25/-1.3, 1/-120/0.53/2.29, -1/-120/2.25/1.3}
{
\begin{scope}[shift={(\x cm, \y cm)}, scale=0.5, shift={(\c cm, \d cm)}, rotate=\b, yscale=\a]

\draw
	(-1.5,0) -- (240:1.5) -- (-60:1.5) -- (0:1.5) -- ++(120:1.1453) -- ++(-1.1453,0) -- (-1.5,0);

\node at (1.2,0) {\tiny 0};
\node at (0.8,0.7) {\tiny 1};
\node at (-0.2,0.7) {\tiny 2};
\node at (-1.17,-0.02) {\tiny 3};
\node at (-120:1.2) {\tiny 4};
\node at (-60:1.2) {\tiny 5};

\end{scope}
}

\foreach \x/\y/\a/\b/\c/\d in 
{4.5/-4/1/180/0/-2.6, 4.5/-4/1/-120/2.25/-1.3,
9/-4/1/180/0/-2.6, 9/-4/-1/60/2.25/-1.3, 
4.5/0/-1/0/0/-2.6, 4.5/0/1/-120/2.25/-1.3,
9/0/-1/60/2.25/-1.3, 9/0/-1/0/0/-2.6,
0/-4/-1/60/0/-2.6, 0/-4/1/180/2.25/-1.3}
{
\begin{scope}[shift={(\x cm, \y cm)}, scale=0.5, shift={(\c cm, \d cm)}, rotate=\b, yscale=\a]

\draw
	(-1.5,0) -- (240:1.5) -- (-60:1.5) -- (0:1.5) -- ++(120:1.1453) -- ++(-1.1453,0) -- (-1.5,0);

\node at (1.2,0) {\tiny 0};
\node at (0.8,0.7) {\tiny 1};
\node at (-0.2,0.7) {\tiny 2};
\node at (-1.17,-0.02) {\tiny 3};
\node at (-120:1.2) {\tiny 4};
\node at (-60:1.2) {\tiny 5};

\end{scope}
}

\foreach \x/\y in {4.5/0, 9/0, 0/-4, 4.5/-4, 9/-4}
{
\begin{scope}[shift={(\x cm, \y cm)}, scale=0.5]

\fill
	(1.5,2) circle (0.1);

\node at (1.6,2.8) {\tiny $a$};
\node at (2.2,2.5) {\tiny $a$};
\node at (1.7,2.4) {\tiny $\lambda$};

\end{scope}
}

\end{tikzpicture}
\caption{$H_{8.1}$, and five ${\mc D}_1$-tilings.}
\label{H8A}
\end{figure}
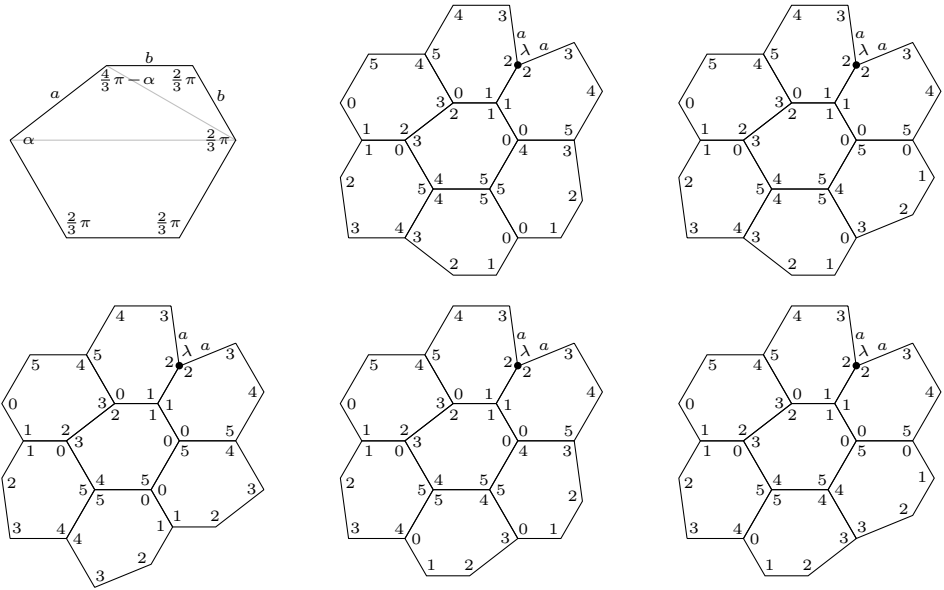

Figure \ref{H8A} also shows five ${\mc D}_1$-tilings $\tilde{0}130\tilde{4}5$, $\tilde{0}130\tilde{4}\tilde{4}$, $\tilde{0}130\tilde{5}0$, $\tilde{0}13055$, $\tilde{0}1305\tilde{4}$ by $H_{8.1}$. These are real ${\mc D}_1$-tilings only if all angles are positive, and $a,b>0$, and the tiles in the ${\mc D}_1$-tilings do not overlap. The first two conditions mean $\frac{1}{3}\pi<\alpha<\frac{7}{6}\pi$. The third condition implies the sum of the two angles at any degree $3$ vertex of a ${\mc D}_1$-tiling is $<2\pi$. For example, for $\tilde{0}130\tilde{4}5$, this means the six sums $[3]+[5]$, $2[2]$, $[4]+[5]$, $2[1]$, $[3]+[4]$, $2[0]$ are $<2\pi$. It is easy to see that all six sums are indeed $<2\pi$ for $\frac{1}{3}\pi<\alpha<\frac{7}{6}\pi$. By similar reason, we see that the first four ${\mc D}_1$-tilings $\tilde{0}130\tilde{4}5$, $\tilde{0}130\tilde{4}\tilde{4}$, $\tilde{0}130\tilde{5}0$, $\tilde{0}13055$ are real tilings for $\frac{1}{3}\pi<\alpha<\frac{7}{6}\pi$. In the fifth ${\mc D}_1$-tiling $\tilde{0}1305\tilde{4}$, we require $2[3]<2\pi$, which means $\alpha<\pi$. Then $\tilde{0}1305\tilde{4}$ is a real tiling for $\frac{1}{3}\pi<\alpha<\pi$.

If the ${\mc D}_1$-tilings are extended to ${\mc D}_2$-tilings, then at the degree $3$ vertex $\bullet$, the angle $\lambda$ between the two $a$ edges is an angle of the hexagon, and satisfies $2(\frac{4}{3}\pi-\alpha)+\lambda=2\pi$. This implies $a=b$ or $a=1$, and $\lambda=\frac{2}{3}\pi,\alpha,\frac{4}{3}\pi-\alpha$. In all cases, the equation $2(\frac{4}{3}\pi-\alpha)+\lambda=2\pi$ implies $\alpha=\frac{2}{3}\pi$. Then by the formulae for $a$ and $b$, we get $a=b=1$. Therefore the hexagon is regular.

We remark that, as a special case of $H_{7.1}$, we also have the ${\mc D}_1$-tilings $213054$ or $2130\tilde{4}4$ by $H_{8.1}$. However, the two tilings do not imply the extra equality $[0]=\frac{2}{3}\pi$ satisfied by $H_{8.1}$.

\medskip

\noindent{\em Hexagon $H_{8.2}$}

\medskip

In addition to $[1]=[3]=[5]=\frac{2}{3}\pi$, we set 
\[
|\bar{1}|=|\bar{3}|=|\bar{4}|=|\bar{5}|=1,\;
|\bar{0}|=a,\;
|\bar{2}|=b,\;
[0]=[2]=\alpha,\;
[4]=2\pi-2\alpha.
\]
Substituting into \eqref{hexeq}, we get
\[
a
=\frac{4\sqrt{3}\sin^2\alpha-2\sin\alpha-\sqrt{3}}{2\sin(\alpha-\tfrac{1}{3}\pi)},\quad
b=\frac{4\sin\alpha-\sqrt{3}}{2\sin(\alpha-\tfrac{1}{3}\pi)}.
\]
The specific hexagon $H_{8.2}$ in Figure \ref{H8B} has
\[
\alpha=130^{\circ},\quad
a=0.42647,\quad
b=0.70881,
\]
and is not a Reinhardt hexagon.

\begin{figure}[htp]
\centering
\begin{tikzpicture}[>=latex,scale=1]


\begin{scope}[scale=1.5]

\draw[gray!50]
	(0,0) -- (0.366,1.45) -- (1.5,0.863) -- cycle;
	
\draw
	(0,0) -- ++(0:1) -- ++(60:1) -- ++(110:0.4365) -- ++(170:1) -- ++(220:0.7088) -- ++(280:1);

\node at (1.5,1.1) {\tiny $a$};
\node at (0.06,1.28) {\tiny $b$};

\node at (0.9,0.12) {\tiny $\frac{2}{3}\pi$};
\node at (1.25,1.18) {\tiny $\frac{2}{3}\pi$};
\node at (-0.02,0.92) {\tiny $\frac{2}{3}\pi$};
\node at (1.4,0.87) {\tiny $\alpha$};
\node at (0.39,1.35) {\tiny $\alpha$};
\node at (0.28,0.1) {\tiny $2\pi\!-\!2\alpha$};
\end{scope}

\begin{scope}[shift={(4.5cm,0.8cm)}, scale=0.8]

\foreach \a/\b/\c/\d in {1/0/0/0, -1/-20/-1.12/1.33, 1/120/1.5/-0.87, 1/-120/1.5/0.87, -1/-140/3.21/1.17, 1/-10/0.36/1.45, -1/80/-1.13/1.34}
{
\begin{scope}[shift={(\c cm, \d cm)}, rotate=\b, yscale=\a]

\draw
	(0,0) -- ++(0:1) -- ++(60:1) -- ++(110:0.4365) -- ++(170:1) -- ++(220:0.7088) -- ++(280:1);

\node at (1.3,0.85) {\tiny 0};
\node at (1.2,1.15) {\tiny 1};
\node at (0.4,1.25) {\tiny 2};
\node at (0,0.9) {\tiny 3};
\node at (0.15,0.17) {\tiny 4};
\node at (0.9,0.17) {\tiny 5};

\end{scope}
}

\fill
	(2,2.04) circle (0.07);
	
\node at (2.05,2.35) {\tiny $a$};	
\node at (2.45,2.2) {\tiny $b$};	
\node at (2.2,2.2) {\tiny $\lambda$};

\end{scope}

\end{tikzpicture}
\caption{$H_{8.2}$, and a ${\mc D}_1$-tiling.}
\label{H8B}
\end{figure}
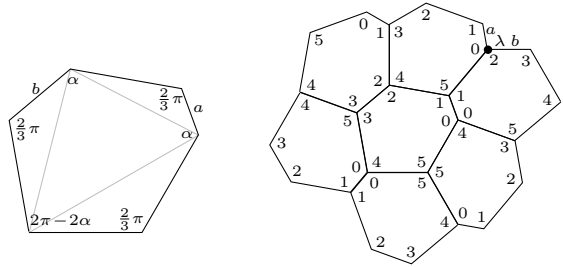

Figure \ref{H8B} also shows the ${\mc D}_1$-tiling $\tilde{0}5\tilde{2}\tilde{5}05$ by $H_{8.2}$. For all angles to be positive, we need $0<\alpha<\pi$. For the tiles in the ${\mc D}_1$-tiling not to overlap, we need $2[4]<2\pi$. This means $\alpha>\frac{1}{2}\pi$. For $\frac{1}{2}\pi<\alpha<\pi$, we know the denominator $2\sin(\alpha-\tfrac{1}{3}\pi)$ of $a$ and $b$ are positive. Then we need both numerators $4\sqrt{3}\sin^2\alpha-2\sin\alpha-\sqrt{3}$ and $4\sin\alpha-\sqrt{3}$ to be positive. This means $\frac{1}{2}\pi<\alpha<\pi-\arcsin\frac{1+\sqrt{13}}{4\sqrt{3}}=138.33654^{\circ}$. Then we find $\tilde{0}5\tilde{2}\tilde{5}05$ is a real ${\mc D}_1$-tiling for $\alpha$ within this range.

If the ${\mc D}_1$-tiling is extended to a ${\mc D}_2$-tiling, then the angle $\lambda$ at the degree $3$ vertex $\bullet$ is an angle of the hexagon and satisfies $\alpha+\alpha+\lambda=2\pi$. Since $a$ and $b$ are separated in $H_{8.2}$, and $\lambda$ is between $a$ and $b$, we get $a=1$ or $b=1$. If $a\ne 1$ and $b=1$, or $a=1$ and $b\ne 1$, then $\lambda=\alpha,\frac{2}{3}\pi$. Substituting into the equation $\alpha+\alpha+\lambda=2\pi$, we get $\alpha=\frac{2}{3}\pi$. Then the formulae for $a$ and $b$ imply $a=b=1$.

Therefore we have $a=b=1$. Then we get three isosceles  triangles of side length $1$ and top angle $\frac{2}{3}\pi$, with the gray edges as bases. The triangles are congruent, and the gray edges form a regular triangle. This implies that the hexagon is regular.

\medskip

\noindent{\em Hexagon $H_{8.3}$}

\medskip

We set 
\[
|\bar{0}|=|\bar{1}|=|\bar{3}|=|\bar{5}|=1,\;
|\bar{2}|=a, \;
|\bar{4}|=b,
\]
\[
[0]=[2]=[4]=2\pi-2\alpha,\;
[1]=4\alpha-2\pi,\;
[3]=[5]=\alpha.
\]
Substituting into \eqref{hexeq}, we get
\[
a=4\cos^2\alpha+2\cos\alpha+1,\quad
b=-4\cos\alpha-1.
\]
The specific hexagon $H_{8.3}$ in Figure \ref{H8A} has
\[
\alpha=110^{\circ},\;
a=0.78387,\;
b=0.36808,
\]
and is not a Reinhardt hexagon.

\begin{figure}[htp]
\centering
\begin{tikzpicture}[>=latex,scale=1]


\begin{scope}[scale=1.5]

\draw
	(0,0) -- ++(1,0) -- ++(40:1) -- ++(140:1) -- ++(180:0.78387) -- ++(250:1) -- ++(290:0.36808);

\node at (0.6,1.35) {\tiny $a$};
\node at (-0.12,0.17) {\tiny $b$};

\node at (0.08,0.1) {\tiny $\alpha$};
\node at (0.8,0.1) {\tiny $2\pi\!-\!2\alpha$};
\node at (0.2,0.35) {\tiny $2\pi\!-\!2\alpha$};
\node at (0.8,1.2) {\tiny $2\pi\!-\!2\alpha$};
\node at (0.27,1.2) {\tiny $\alpha$};
\node at (1.4,0.63) {\tiny $4\alpha\!-\!2\pi$};

\end{scope}

\begin{scope}[shift={(5cm,0.6cm)}, scale=0.8]

\foreach \a/\b/\c/\d in {1/0/0/0, 1/180/1/0, -1/0/0/2.57, -1/-70/-0.13/2.227, -1/-140/0/0, -1/40/1/0, -1/180/2.77/0.645}
{
\begin{scope}[shift={(\c cm, \d cm)}, rotate=\b, yscale=\a]

\draw
	(0,0) -- ++(1,0) -- ++(40:1) -- ++(140:1) -- ++(180:0.78387) -- ++(250:1) -- ++(290:0.36808);

\node at (0.95,0.17) {\tiny 0};
\node at (1.5,0.65) {\tiny 1};
\node at (0.95,1.12) {\tiny 2};
\node at (0.33,1.12) {\tiny 3};
\node at (0.05,0.4) {\tiny 4};
\node at (0.12,0.17) {\tiny 5};

\end{scope}
}

\fill
	(-0.99,0.84) circle (0.07);

\filldraw[fill=white]
	(-0.13,2.23) circle (0.07);
	
\node at (-1.25,1.2) {\tiny $a$};
\node at (-1.35,0.74) {\tiny $a$};

\node at (-0.15,2.5) {\tiny $b$};
\node at (-0.35,2.4) {\tiny $b$};
	
\end{scope}

\end{tikzpicture}
\caption{$H_{8.3}$, and a ${\mc D}_1$-tiling.}
\label{H8C}
\end{figure}
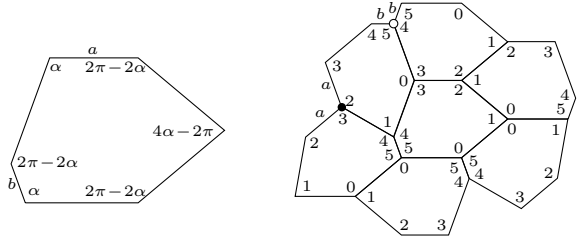

Figure \ref{H8C} also shows the ${\mc D}_1$-tiling $\tilde{5}\tilde{0}\tilde{2}\tilde{0}\tilde{4}0$ by $H_{8.3}$. For all angles to be positive, we need $\frac{1}{2}\pi<\alpha<\pi$. For the tiles in the ${\mc D}_1$-tiling not to overlap, we need $[1]+[5]<2\pi$. This means $\alpha<\frac{4}{5}\pi$. We always have $a>0$. Within the range $\frac{1}{2}\pi<\alpha<\frac{4}{5}\pi$, $b>0$ means $\frac{4}{5}\pi>\alpha>\frac{1}{2}\pi+\arcsin\frac{1}{4}=104.47751^{\circ}$. Then we find $\tilde{5}\tilde{0}\tilde{2}\tilde{0}\tilde{4}0$ is a real ${\mc D}_1$-tiling for $\alpha$ within this range.

If the ${\mc D}_1$-tiling is extended to a ${\mc D}_2$-tiling, then at the degree $3$ vertices $\bullet$ and $\circ$, we have the angles between two $a$ and between two $b$ that are also angles in the hexagon. Since $a$ and $b$ are separated in $H_{8.3}$, this implies $a=b=1$. Substituting into the formulae for $b$, we get $\cos\alpha=-\frac{1}{2}$. On the other hand, for all angles of $H_{8.3}$ to be positive, we need $\frac{\pi}{2}<\alpha<\pi$. Therefore $\alpha=\frac{2}{3}\pi$, and the hexagon is regular.

\medskip

\noindent{\em Hexagon $H_{8.4}$}

\medskip

In addition to $[4]=[5]=\frac{2}{3}\pi$, we set 
\[
|\bar{1}|=|\bar{3}|=|\bar{4}|=|\bar{5}|=1,\;
|\bar{0}|=|\bar{2}|=a,
\]
\[
[0]=\alpha,\;
[1]=\beta,\;
[2]=\tfrac{4}{3}\pi-\beta,\;
[3]=\tfrac{4}{3}\pi-\alpha.
\]
For all angles to be positive, we need $0<\alpha,\beta<\frac{4}{3}\pi$. 

Substituting the setting into \eqref{hexeq}, we get
\[
2e^{i(\alpha-\frac{2}{3}\pi)}-a=e^{i(\frac{2}{3}\pi-\beta)}.
\]
We rewrite the equation as
\[
2e^{i\alpha'}-a=e^{i\beta'},\quad
\alpha'=\alpha-\tfrac{2}{3}\pi,\quad
\beta'=\tfrac{2}{3}\pi-\beta.
\]
The range $0<\alpha,\beta<\frac{4}{3}\pi$ becomes $-\frac{2}{3}\pi<\alpha',\beta'<\frac{2}{3}\pi$, and the equation is described in Figure \ref{H8F}. The circles have radii $1$ and $2$, indicating the locations of $e^{i\beta'}$ and $2e^{i\alpha'}$. The horizontal lines indicate $2e^{i\alpha'}-a$ for all $a>0$. 

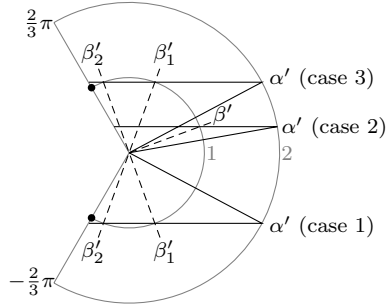
\begin{figure}[htp]
\centering
\begin{tikzpicture}[>=latex,scale=1]

\draw[gray]
	(120:1) arc (120:-120:1)
	(120:2) arc (120:-120:2)
	(-120:2) -- (0,0) -- (120:2);

\draw
	(0,0) -- (10:2) -- ++(-2.17,0)
	(0,0) -- (28:2) -- ++(-2.3,0)
	(0,0) -- (-28:2) -- ++(-2.3,0);

\draw[densely dashed]
	(70:-1.2) -- (70:1.2)
	(110:-1.2) -- (110:1.2)
	(0,0) -- (20.3:1.2);

\fill
	(120:1) circle (0.05)
	(-120:1) circle (0.05);

\node at (125:2.1) {\scriptsize $\frac{2}{3}\pi$};
\node at (-127:2.15) {\scriptsize $-\frac{2}{3}\pi$};

\node at (21:2.75) {\scriptsize $\alpha'$  (case 3)};
\node at (-21:2.75) {\scriptsize $\alpha'$  (case 1)};
\node at (70:1.4) {\scriptsize $\beta'_1$};
\node at (110:1.4) {\scriptsize $\beta'_2$};
\node at (-70:1.4) {\scriptsize $\beta'_1$};
\node at (-110:1.4) {\scriptsize $\beta'_2$};
\node at (7:2.75) {\scriptsize $\alpha'$ (case 2)};
\node at (22:1.35) {\scriptsize $\beta'$};

\node[gray] at (1.1,0) {\scriptsize $1$};
\node[gray] at (2.1,0) {\scriptsize $2$};

\end{tikzpicture}
\caption{Calculation for $H_{8.4}$.}
\label{H8F}
\end{figure}

The solution is the intersection of the horizontal line with the circle of radius $1$. The intersection is not empty if and only if 
\[
-\tfrac{1}{6}\pi\le \alpha'\le \tfrac{1}{6}\pi,\quad
\tfrac{1}{2}\pi\le\alpha\le\tfrac{5}{6}\pi.
\]
Moreover, if the horizontal line is between two $\bullet$, then there is only one solution $\beta'$ within the range. If the horizontal line is not between two $\bullet$, then there are two solutions $\beta_1',\beta_2'$ within the range. 

The transition $\bullet$ between single $\beta$ and double $\beta$ happens when $\beta'=\frac{2}{3}\pi$ or $-\frac{2}{3}\pi$ is a solution. This means $2e^{i\alpha'}-a=e^{\pm i\frac{2}{3}\pi}$. Taking the imaginary part, we get $2\sin \alpha'=\sin(\pm\frac{2}{3}\pi)=\pm\frac{\sqrt{3}}{2}$. Then we get the angles corresponding to the two $\bullet$
\[
\alpha=\tfrac{2}{3}\pi\pm\theta
=94.34109^{\circ},145.65891^{\circ};\quad
\theta=\arcsin\tfrac{\sqrt{3}}{4}=25.65891^{\circ}.
\]

The range $\tfrac{1}{2}\pi\le\alpha\le\tfrac{5}{6}\pi$ is divided by $\tfrac{2}{3}\pi\pm\theta$ into three cases:
\begin{enumerate}
\item $\frac{1}{2}\pi\le\alpha<\tfrac{2}{3}\pi-\theta$: There are two solutions $\beta_1,\beta_2$ satisfying $\pi<\beta_1\le \beta_2<\frac{4}{3}\pi$ and $\beta_1+\beta_2=\frac{7}{3}\pi$. The first and second of Figure \ref{H8D} show the hexagons corresponding to the two solutions. The two solutions merge into one (i.e., two hexagons are the same) for $\alpha=\frac{1}{2}\pi$ and $\beta=\frac{7}{6}\pi$. 
\item $\tfrac{2}{3}\pi-\theta\le\alpha\le \tfrac{2}{3}\pi+\theta$: There is only one solution $\beta$ satisfying $\frac{1}{3}\pi\le\beta\le \pi$. The corresponding hexagon is the third of Figure \ref{H8D}, and is always convex.
\item $\tfrac{2}{3}\pi+\theta<\alpha\le \frac{5}{6}\pi$: There are two solutions $\beta_1,\beta_2$ satisfying $0<\beta_2\le \beta_1<\frac{1}{3}\pi$ and $\beta_1+\beta_2=\frac{1}{3}\pi$. The corresponding hexagons are the horizontal flips of the first and second of Figure \ref{H8D}. The two solutions merge into one for $\alpha=\frac{5}{6}\pi$ and $\beta=\frac{1}{6}\pi$. 
\end{enumerate}
We note that the exchange between $(\alpha,\beta)$ and $(\frac{4}{3}\pi-\alpha,\frac{4}{3}\pi-\beta)$ preserves the hexagon $H_{8.4}$ as a family, and exchanges the first and third cases.

The length $a$ is the distance between the point $\alpha'$ on the circle of radius $2$ and the point $\beta'$ on the circle of radius $1$. We may use $|2e^{i(\alpha-\frac{2}{3}\pi)}-a|=1$ to calculate the formula for $a$ 
\[
a=2\cos(\alpha-\tfrac{2}{3}\pi)\pm\sqrt{4\cos^2(\alpha-\tfrac{2}{3}\pi)-3}.
\]
The smaller $a$, with $-$, corresponds to $\beta_1$. The bigger $a$, with $+$, corresponds to $\beta_2$. In case of single $\beta$, only the smaller $a$ is used.

Figure \ref{H8D} shows $H_{8.4}$ for $\alpha=93^{\circ}$ and $100^{\circ}$. There are two hexagons for $\alpha=93^{\circ}$
\begin{align*}
\alpha &=93^{\circ}, \quad
a =1.36300, \quad
\beta =185.22781^{\circ}; \\
\alpha &=93^{\circ}, \quad
a =2.20102, \quad
\beta =234.77218^{\circ}.
\end{align*}
There is only one hexagon for $\alpha=100^{\circ}$
\[
\alpha=100^{\circ},\quad
a=1.14994, \quad
\beta=163.16018^{\circ}.
\]
All are not Reinhardt hexagons.

\begin{figure}[htp]
\centering
\begin{tikzpicture}[>=latex,scale=1.2]


\foreach \a in {0,1,2}
\draw[gray!50, xshift=3*\a cm]
	(-1,0) -- (1,0);

\draw
	(180:1) -- (240:1) -- (-60:1) -- (0:1) -- ++(147:1.363) -- ++(141.7722:1) -- ++(267:1.363);

\draw[xshift=3cm]
	(180:1) -- (240:1) -- (-60:1) -- (0:1) -- ++(147:2.201) -- ++(92.2289:1) -- ++(267:2.201);
	
\draw[xshift=6cm]
	(180:1) -- (240:1) -- (-60:1) -- (0:1) -- ++(140:1.15) -- ++(156.84:1) -- ++(260:1.15);

\node at (0.45,0.45) {\tiny $a$};
\node at (-1.05,0.6) {\tiny $a$};

\node at (0.85,-0.02) {\tiny $\alpha$};
\node at (-0.2,0.65) {\tiny $\beta$};
\node at (-60:0.8) {\tiny $\frac{2}{3}\pi$};
\node at (240:0.8) {\tiny $\frac{2}{3}\pi$};
\node at (-0.65,0.05) {\tiny $\frac{4}{3}\pi\!-\!\alpha$};
\node[rotate=-60] at (-0.68,0.87) {\tiny $\frac{4}{3}\pi\!-\!\beta$};

\node at (0,-0.3) {\footnotesize$\alpha=93^{\circ}$};
\node at (3,-0.3) {\footnotesize$\alpha=93^{\circ}$};
\node at (6,-0.3) {\footnotesize$\alpha=100^{\circ}$};

\node at (3.1,0.7) {\tiny $a$};
\node at (1.95,1) {\tiny $a$};

\node at (6.6,0.45) {\tiny $a$};
\node at (5,0.6) {\tiny $a$};

\end{tikzpicture}
\caption{$H_{8.4}$, non-indicated edges have length 1.}
\label{H8D}
\end{figure}
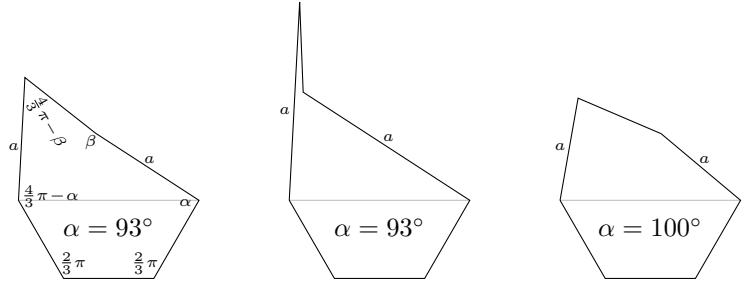

Figure \ref{H8E} describes the 20 ${\mc D}_1$-tilings leading to the hexagon $H_{8.4}$, with $\alpha=100^{\circ}$. We label a first layer tile by the label of the edge of the gray center tile. Moreover, we may assign $\pm$ to the first layer tile by comparing its orientation with the orientation of the center tile.

Since $H_{8.4}$ is symmetric with respect to the horizontal flip, i.e., the exchange between $(\alpha,\beta)$ and $(\frac{4}{3}\pi-\alpha,\frac{4}{3}\pi-\beta)$, we should consider a ${\mc D}_1$-tiling and its horizontal flip as the same tiling. 

The ${\mc D}_1$-tilings can be divided into the upper part {\circled 0} {\circled 1} {\circled 2} and the lower part {\circled 3} {\circled 4} {\circled 5}. The two parts can be independently chosen.

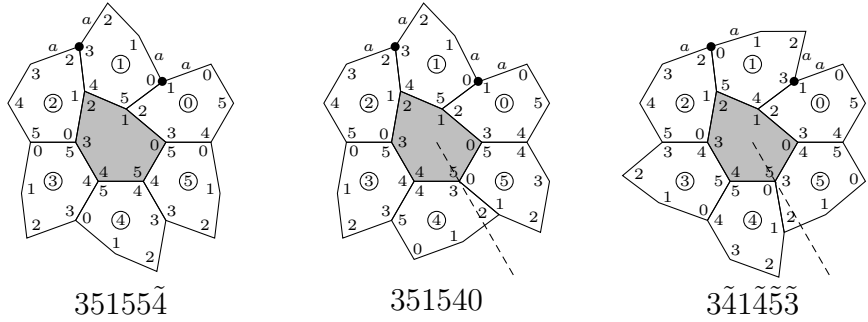
\begin{figure}[htp]
\centering
\begin{tikzpicture}[>=latex,scale=0.6]

\foreach \x in {0,1,2}
\fill[gray!50, xshift=7*\x cm]
	(180:1) -- (240:1) -- (-60:1) -- (0:1) -- ++(140:1.15) -- ++(156.84:1) -- ++(260:1.15);	

\foreach \a/\b/\c/\d in 
{1/0/0/0, 1/-60/-1.5/0.866, 1/60/1.5/0.866, 1/-23.2/0/1.73, 1/120/-1.5/-0.866, 1/180/0/-1.732, -1/60/1.5/-0.866,
1/0/7/0, 1/-60/5.5/0.866, 1/60/8.5/0.866, 1/-23.2/7/1.73, 
1/120/5.5/-0.866, 1/-120/7/-1.732, 1/180/8.5/-0.866,
1/0/14/0, 1/-60/12.5/0.866, 1/60/15.5/0.866, -1/157/14/1.73,
-1/-60/12.5/-0.866, -1/60/14/-1.732, -1/0/15.5/-0.866 }
{
\begin{scope}[shift={(\c cm, \d cm)}, rotate=\b, yscale=\a]

\draw
	(180:1) -- (240:1) -- (-60:1) -- (0:1) -- ++(140:1.15) -- ++(156.84:1) -- ++(260:1.15);
	
\node at (0.75,-0.05) {\tiny 0};
\node at (0.1,0.5) {\tiny 1};
\node at (-0.65,0.8) {\tiny 2};
\node at (-0.75,0) {\tiny 3};
\node at (-120:0.75) {\tiny 4};
\node at (-60:0.75) {\tiny 5};

\end{scope}
}

\foreach \x in {1,2}
\draw[densely dashed]
	(7*\x,0) -- ++(-60:3.5);

\foreach \x in {0,...,5}
\foreach \z in {0,1,2}
{
\node[xshift=4.2*\z cm, inner sep=0.5, draw, shape=circle] at (30+60*\x:1.73) {\tiny \x};
}

\node at (0,-3.5) {$35155\tilde{4}$};
\node at (7,-3.5) {$351540$};
\node at (14,-3.5) {$3\tilde{4}1\tilde{4}\tilde{5}\tilde{3}$};

\foreach \x in {0,1}
{
\begin{scope}[xshift=7*\x cm]

\fill
	(-0.92,2.12) circle (0.1)
	(0.92,1.35) circle (0.1);
	
\node at (-1.5,2.1) {\tiny $a$};
\node at (-0.75,2.7) {\tiny $a$};

\node at (1.45,1.7) {\tiny $a$};
\node at (0.85,1.9) {\tiny $a$};

\end{scope}
}

\begin{scope}[xshift=14 cm]

\fill
	(-0.92,2.12) circle (0.1)
	(0.92,1.35) circle (0.1);
	
\node at (-1.5,2.1) {\tiny $a$};
\node at (-0.4,2.45) {\tiny $a$};

\node at (1.5,1.7) {\tiny $a$};
\node at (1.2,1.9) {\tiny $a$};

\end{scope}
	
\end{tikzpicture}
\caption{$H_{8.4}$, and ${\mc D}_1$-tilings.}
\label{H8E}
\end{figure}

In the upper part, {\circled 0} and {\circled 2} are the same in all the ${\mc D}_1$-tilings, and {\circled 1} has two orientations: $+$ in the first and second of Figure \ref{H8E}, and $-$ in the third  of Figure \ref{H8E}.

The lower part has two possible configurations. The first of Figure \ref{H8E} shows the {\em normal} configuration, where the half regular hexagonal parts of {\circled 3}, {\circled 4}, {\circled 5} are glued to the center tile. In other words, the edge $\bar{4}$ of these tiles are glued to the center tile. Then each of the three tiles can be independently flipped. In other words, their signs can be any combinations of $+$ and $-$. Up to the horizontal flip, we may only consider the sign combination of {\circled 3} and {\circled 5} to be $(+,+),(+,-),(-,-)$.

The second and the third of Figure \ref{H8E} show the {\em skewed} configuration of the lower part, in which two neighboring tiles {\circled 4} and {\circled 5} are glued to the center tile along the edges $\bar{3}$ and $\bar{5}$ instead of $\bar{4}$. The two pairs are related by the flip with respect to the dashed line. We may denote them by their orientations $(+,+)$ and $(-,-)$. 

The remaining tile {\circled 3} is still glued along $\bar{4}$, and can be flipped to give $+$ in the second picture and $-$ in the third picture. Up to the horizontal flip, we may assume the neighboring tiles are {\circled 4} and {\circled 5}, and the remaining tile is {\circled 3}. Then there are total of four combinations of  the lower part.

Table \ref{table8.4A} gives the codes for the upper and lower parts. The whole code for the ${\mc D}_1$-tilings is the code of the upper part followed by the code of the lower part. There are two upper codes and ten lower codes, and we get total of twenty ${\mc D}_1$-tilings.

\begin{table}[htp]
\centering
\begin{tabular}{|c|c|c|c|}
\hline
\multicolumn{4}{|c|}{upper part} \\
\hline \hline
{\circled 0} & {\circled 1} & {\circled 2} & code \\
\hline
$+$ & $+$ & $+$ & $351$ \\
$+$ & $-$ & $+$ & $3\tilde{4}1$ \\
\hline
\end{tabular}
\quad
\begin{tabular}{|c|c|c|c|}
\hline
\multicolumn{4}{|c|}{normal lower part} \\
\hline \hline
{\circled 3} & {\circled 4} & {\circled 5} & code \\
\hline
$+$ & $+$ & $+$ & $555$ \\
$+$ & $-$ & $+$ & $5\tilde{4}5$ \\
$+$ & $+$ & $-$ & $55\tilde{4}$ \\
$+$ & $-$ & $-$ & $5\tilde{4}\tilde{4}$ \\
$-$ & $+$ & $-$ & $\tilde{4}55$ \\
$-$ & $-$ & $-$ & $\tilde{4}\tilde{4}5$ \\
\hline
\end{tabular}
\quad
\begin{tabular}{|c|c|c|c|}
\hline
\multicolumn{4}{|c|}{skewed lower part} \\
\hline \hline
{\circled 3} & {\circled 4} & {\circled 5} & code \\
\hline
$+$ & $+$ & $+$ & $540$ \\
$+$ & $-$ & $-$ & $5\tilde{5}\tilde{3}$ \\
$-$ & $+$ & $+$ & $\tilde{4}40$ \\
$-$ & $-$ & $-$ & $\tilde{4}\tilde{5}\tilde{3}$ \\
\hline
\end{tabular}
\caption{Code for the ${\mc D}_1$-tilings of $H_{8.4}$.}
\label{table8.4A}
\end{table}

Next, we consider the implication of the second layer of tiles. At the degree $3$ vertices $\bullet$ in the upper part, we have the angles between two $a$ that are also angles in the hexagon. Since two $a$ are separated in $H_{8.4}$, this implies $a=1$. In Figure \ref{H8F}, this means $\alpha'=\pi$, which is the same as $\alpha=\frac{2}{3}\pi$. Then by all edges having length $1$, we conclude the hexagon is regular.

Finally, we remark that some ${\mc D}_1$-tiling may not happen due to overlapping among the tiles. Figure \ref{H8G} is the ${\mc D}_1$-tiling $3\tilde{4}1\tilde{4}\tilde{5}\tilde{3}$ for the second of Figure \ref{H8D}, with $\alpha=93^{\circ}$, $
a=2.20102$, and $\beta =234.77218^{\circ}$. The overlapping between {\circled 0} and {\circled 1} happens due to $[1]+[3]=\frac{4}{3}\pi-\alpha+\beta=381.77218^{\circ}>2\pi$. There is no overlapping among the tiles in the single $\beta$ case, because the hexagon is convex. Therefore we only need to consider the first case $\frac{1}{2}\pi\le\alpha<\tfrac{2}{3}\pi-\theta$ and the third case $\tfrac{2}{3}\pi+\theta<\alpha\le \frac{5}{6}\pi$.

\begin{figure}[htp]
\centering
\begin{tikzpicture}[>=latex,scale=0.5]

\foreach \a/\b/\c/\d in {
1/0/0/0, 1/-60/-1.5/0.866, 1/60/1.5/0.866, -1/92/0/1.73,
-1/-60/-1.5/-0.866, -1/60/0/-1.732, -1/0/1.5/-0.866}
\draw[shift={(\c cm, \d cm)}, rotate=\b, yscale=\a]
	(180:1) -- (240:1) -- (-60:1) -- (0:1) -- ++(147:2.201) -- ++(92.2289:1) -- ++(267:2.201);

\end{tikzpicture}
\caption{Overlapping for a ${\mc D}_1$-tiling by $H_{8.4}$.}
\label{H8G}
\end{figure}
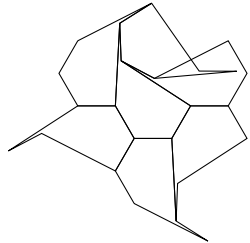

Now we examine all the sum of two angles along the boundary of ${\mc D}_1$-tilings that might be larger than $2\pi$. First, we use $\frac{1}{2}\pi\le \alpha\le \frac{5}{6}\pi$ to get
\begin{align*}
[0]+[0]
&=2\alpha\le\tfrac{5}{3}\pi<2\pi, \\
[3]+[3]
&=\tfrac{8}{3}\pi-2\alpha\le\tfrac{5}{3}\pi<2\pi, \\
[0]+[4]=[0]+[5]
&=\alpha+\tfrac{2}{3}\pi\le\tfrac{3}{2}\pi<2\pi, \\
[3]+[4]=[3]+[5]
&=2\pi-\alpha\le \tfrac{3}{2}\pi<2\pi, \\
[0]+[3]=[1]+[2]=[4]+[4]=[4]+[5]
&=\tfrac{4}{3}\pi<2\pi.
\end{align*}
Then we find that, in all twenty ${\mc D}_1$-tilings, the only two angle combinations on the boundary that are not in the inequalities above are the two $\bullet$ in Figure \ref{H8E}. 

For the upper part $351$ (the first and second of Figure \ref{H8E}), we have the angle sums on the boundary
\[
[0]+[1]=\alpha+\beta,\quad
[2]+[3]=\tfrac{8}{3}\pi-\alpha-\beta.
\]
The requirement $[0]+[1]\le 2\pi$ and $[2]+[3]\le 2\pi$ means $|\alpha'-\beta'|\le \frac{2}{3}\pi$. By Figure \ref{H8F}, this is always true for both $\beta_1'$ and $\beta_2'$.

For the upper part $3\tilde{4}1$ (the third of Figure \ref{H8E}), we have the angle sums on the boundary
\[
[0]+[2]=\tfrac{4}{3}\pi+\alpha-\beta,\quad
[1]+[3]=\tfrac{4}{3}\pi-\alpha+\beta.
\]
The requirement $[0]+[2]\le 2\pi$ and $[1]+[3]\le 2\pi$ means $|\alpha'+\beta'|\le \frac{2}{3}\pi$. By Figure \ref{H8F}, we always have $|\alpha'+\beta_1'|\le \frac{1}{6}\pi+\frac{1}{2}\pi=\frac{2}{3}\pi$, and the equality happens only for the case $\beta_1'$ and $\beta_2'$ become the same. 

Next we study $|\alpha'+\beta_2'|\le \frac{2}{3}\pi$. In case 3, we consider $|\alpha'+\beta_2'|=\alpha'+\beta_2'$ as a continuous function of $\beta_2'\in [\frac{1}{2}\pi,\frac{2}{3}\pi]$. Substituting into the imaginary part of $2e^{i\alpha'}-a=e^{i\beta_2'}$, we get 
\[
\sin\beta_2'=2\sin\alpha'
=2\sin(\tfrac{2}{3}\pi-\beta_2')
=\sqrt{3}\cos\beta_2'+\sin\beta_2'.
\]
The only solution on $[\frac{1}{2}\pi,\frac{2}{3}\pi]$ is $\beta_2'=\frac{1}{2}\pi$. Morever, at the other end $\beta_2'=\frac{2}{3}\pi$ of the interval, we have $\alpha'+\beta_2'>\beta_2'=\frac{2}{3}\pi$. By the intermediate value theorem, this implies $\alpha'+\beta_2'>\frac{2}{3}\pi$ for all $\beta_2'\in (\frac{1}{2}\pi,\frac{2}{3}\pi]$. By similar argument, we also know $|\alpha'+\beta_2'|>\frac{2}{3}\pi$ in case 1.

In summary the hexagon, $H_{8.4}$ has twenty ${\mc D}_1$-tilings $351X$ and $3\tilde{4}1X$ for the choice of smaller $a$, and has only ten ${\mc D}_1$-tilings $351X$ for the choice of bigger $a$.

\subsection{Rank 9}

We carry out the same process for the rank 9 ${\mc D}_1$-tilings, removing those that can imply Reinhardt hexagons. What remain are fourteen hexagons. Since the rank is high enough, we can set one angle to be $\alpha$, and express all the other angles in terms of $\alpha$. We also indicate the range of $\alpha$, such that all angle values are positive. 
\begin{itemize}
\item $H_{9.1}$: 
$243\tilde{3}\tilde{3}\tilde{5}$,
$24\tilde{3}\tilde{4}5\tilde{5}$.

$|\bar{0}|=|\bar{1}|=|\bar{2}|=|\bar{3}|=|\bar{4}|$,

$[0]=[3]=4\alpha-2\pi$,
$[1]=[5]=\alpha$,
$[2]=6\pi-8\alpha$,
$[4]=2\pi-2\alpha$. 
$\frac{1}{2}\pi<\alpha<\frac{3}{4}\pi$.
\item $H_{9.2}$: 
$2\tilde{3}2\tilde{1}2\tilde{5}$,
$\tilde{2}1\tilde{0}\tilde{1}2\tilde{5}$.

$|\bar{0}|=|\bar{1}|=|\bar{2}|=|\bar{3}|=|\bar{4}|$,

$[0]=\alpha$,
$[1]=2\alpha-\frac{2}{3}\pi$,
$[2]=2\pi-2\alpha$,
$[3]=\frac{10}{3}\pi-4\alpha$,
$[4]=4\alpha-2\pi$,
$[5]=\frac{4}{3}\pi-\alpha$. 
$\frac{1}{2}\pi<\alpha<\frac{5}{6}\pi$.
\item $H_{9.3}$: 
$22\tilde{1}\tilde{2}\tilde{3}\tilde{5}$,
$2\tilde{2}\tilde{2}4\tilde{3}\tilde{5}$,
$2\tilde{2}\tilde{2}\tilde{4}5\tilde{5}$.

$|\bar{0}|=|\bar{1}|=|\bar{2}|=|\bar{3}|=|\bar{4}|$,

$[0]=[1]=\alpha$,
$[2]=2\pi-2\alpha$,
$[3]=[5]=4\alpha-2\pi$,
$[4]=6\pi-8\alpha$. $\frac{1}{2}\pi<\alpha<\frac{3}{4}\pi$.
\item $H_{9.4}$: 
$32\tilde{2}43\tilde{5}$.

$|\bar{0}|=|\bar{1}|=|\bar{2}|=|\bar{3}|=|\bar{4}|$,

$[0]=\alpha$,
$[1]=\frac{10}{3}\pi-4\alpha$,
$[2]=2\alpha-\frac{2}{3}\pi$,
$[3]=2\pi-2\alpha$,
$[4]=4\alpha-2\pi$,
$[5]=\frac{4}{3}\pi-\alpha$. 
$\frac{1}{2}\pi<\alpha<\frac{5}{6}\pi$.
\item $H_{9.5}$: 
$\tilde{2}\tilde{3}23\tilde{2}\tilde{5}$,
$\tilde{2}\tilde{3}\tilde{2}\tilde{1}\tilde{2}\tilde{5}$.

$|\bar{0}|=|\bar{1}|=|\bar{2}|=|\bar{3}|=|\bar{4}|$,

$[0]=\alpha$,
$[1]=\frac{10}{3}\pi-4\alpha$,
$[2]=2\pi-2\alpha$,
$[3]=2\alpha-\frac{2}{3}\pi$,
$[4]=4\alpha-2\pi$,
$[5]=\frac{4}{3}\pi-\alpha$. 
$\frac{1}{2}\pi<\alpha<\frac{5}{6}\pi$.
\item $H_{9.6}$: 
$\tilde{4}54\tilde{3}2\tilde{5}$.

$|\bar{0}|=|\bar{1}|=|\bar{2}|=|\bar{3}|=|\bar{4}|$,

$[0]=\alpha$,
$[1]=\frac{14}{3}\pi-6\alpha$,
$[2]=4\alpha-2\pi$,
$[3]=\frac{2}{3}\pi$,
$[4]=2\pi-2\alpha$,
$[5]=3\alpha-\frac{4}{3}\pi$. 
$\frac{1}{2}\pi<\alpha<\frac{7}{9}\pi$.
\item $H_{9.7}$: 
$32140\tilde{5}$,
$3214\tilde{4}5$,
$3214\tilde{4}\tilde{5}$,
$3\tilde{1}140\tilde{5}$,
$3\tilde{1}14\tilde{4}5$,
$3\tilde{1}14\tilde{4}\tilde{5}$.

$|\bar{0}|=|\bar{1}|=|\bar{2}|=|\bar{3}|$,
$|\bar{4}|=|\bar{5}|$,

$[0]=[4]=\alpha$,
$[1]=[2]=[5]=\frac{2}{3}\pi$,
$[3]=2\pi-2\alpha$.
$0<\alpha<\pi$.
\item $H_{9.8}$: 
$\tilde{3}\tilde{0}1\tilde{0}0\tilde{5}$.

$|\bar{0}|=|\bar{1}|=|\bar{2}|=|\bar{3}|$,
$|\bar{4}|=|\bar{5}|$,

$[0]=\alpha$,
$[1]=\frac{4}{3}\pi-\alpha$,
$[2]=2\alpha-\frac{2}{3}\pi$,
$[3]=2\pi-2\alpha$,
$[4]=[5]=\frac{2}{3}\pi$.
$\frac{1}{3}\pi<\alpha<\pi$.

\item $H_{9.9}$: 
$\tilde{2}32024$,
$1\tilde{0}2024$,
$1\tilde{0}20\tilde{1}4$,
$\tilde{2}320\tilde{1}4$,
$\tilde{2}3\tilde{1}024$,
$\tilde{1}\tilde{2}3024$,
$1\tilde{0}\tilde{1}0\tilde{1}4$,
$\tilde{1}1\tilde{0}0\tilde{1}4$.

$|\bar{0}|=|\bar{1}|=|\bar{2}|=|\bar{4}|$,
$|\bar{3}|=|\bar{5}|$,

$[0]=[1]=[2]=[3]=\frac{2}{3}\pi$,
$[4]=\alpha$, 
$[5]=\frac{4}{3}\pi-\alpha$. 
$0<\alpha<\frac{4}{3}\pi$.
\item $H_{9.10}$: 
$\tilde{1}130\tilde{1}\tilde{5}$,
$1\tilde{0}\tilde{1}0\tilde{1}\tilde{5}$.

$|\bar{0}|=|\bar{1}|=|\bar{2}|=|\bar{4}|$,
$|\bar{3}|=|\bar{5}|$,

$[0]=\alpha$,
$[1]=[5]=2\pi-2\alpha$,
$[2]=[4]=4\alpha-2\pi$,
$[3]=4\pi-5\alpha$.
$\frac{1}{2}\pi<\alpha<\frac{4}{5}\pi$.
\item $H_{9.11}$: 
$21\tilde{0}02\tilde{5}$.

$|\bar{0}|=|\bar{1}|=|\bar{2}|=|\bar{4}|$,
$|\bar{3}|=|\bar{5}|$,

$[0]=\alpha$,
$[1]=[5]=\frac{2}{3}\pi$,
$[2]=2\pi-2\alpha$,
$[3]=\frac{4}{3}\pi-\alpha$,
$[4]=2\alpha-\frac{2}{3}\pi$.
$\frac{1}{3}\pi<\alpha<\pi$.
\item $H_{9.12}$: 
$2\tilde{2}\tilde{0}\tilde{3}24$.

$|\bar{0}|=|\bar{1}|=|\bar{2}|=|\bar{4}|$,
$|\bar{3}|=|\bar{5}|$,

$[0]=4\pi-5\alpha$,
$[1]=[4]=2\pi-2\alpha$,
$[2]=[5]=4\alpha-2\pi$,
$[3]=\alpha$.
$\frac{1}{2}\pi<\alpha<\frac{4}{5}\pi$.
\item $H_{9.13}$: 
$32\tilde{2}540$,
$\tilde{0}21540$,
$3\tilde{1}\tilde{2}540$,
$\tilde{0}2\tilde{2}540$,
$\tilde{0}\tilde{1}1540$,
$\tilde{0}\tilde{1}\tilde{2}540$.

$|\bar{0}|=|\bar{1}|=|\bar{2}|$,
$|\bar{3}|=|\bar{4}|=|\bar{5}|$,

$[0]=[3]=\alpha$,
$[1]=[2]=[4]=\frac{2}{3}\pi$,
$[5]=2\pi-2\alpha$.
$0<\alpha<\pi$.
\item $H_{9.14}$: 
$32\tilde{2}54\tilde{3}$,
$3\tilde{2}3\tilde{3}4\tilde{3}$.

$|\bar{0}|=|\bar{1}|=|\bar{2}|$,
$|\bar{3}|=|\bar{4}|=|\bar{5}|$,

$[0]=\alpha$,
$[1]=\frac{10}{3}\pi-4\alpha$,
$[2]=[5]=2\alpha-\frac{2}{3}\pi$,
$[3]=\frac{4}{3}\pi-\alpha$,
$[4]=\frac{2}{3}\pi$.
$\frac{1}{3}\pi<\alpha<\frac{5}{6}\pi$.
\end{itemize}

Similar to calculations for rank 8 ${\mc D}_1$-tilings, we set the edge lengths in the first edge length equality to be $1$, and the remaining edge lengths to be $a$. For examples, the lengths and angles for $H_{9.1}$ and $H_{9.13}$ are 
\begin{itemize}
\item $H_{9.1}$: $|\bar{0}|=|\bar{1}|=|\bar{2}|=|\bar{3}|=|\bar{4}|=1$ and $|\bar{5}|=a$. $[0]=[1]=\alpha$.
\item $H_{9.13}$: $|\bar{0}|=|\bar{1}|=|\bar{2}|=1$ and $|\bar{3}|=|\bar{4}|=|\bar{5}|=a$. $[0]=[3]=\alpha$.
\end{itemize} 
Then we use \eqref{hexeq} to calculate the hexagon. We find that $H_{9.1},H_{9.2},H_{9.3}$ may not be Reinhardt hexagons, and the hexagons in the remaining eleven cases are always regular.

\medskip

\noindent{\em Hexagon $H_{9.1}$}

\medskip

Substituting the setting above for the edges and angles into \eqref{hexeq}, we get
\[
a+e^{i(3\pi-4\alpha)}+e^{i(4\pi-5\alpha)}+e^{i(-\pi+3\alpha)}+e^{i(2\pi-\alpha)}+e^{i(\pi+\alpha)}=0.
\]
Taking the imaginary part, we get
\begin{align*}
0
&=\sin 4\alpha-\sin 5\alpha-\sin 3\alpha-2\sin\alpha \\
&=-2(2\cos\alpha+1)(4\cos^3\alpha-4\cos^2\alpha+1)\sin\alpha.
\end{align*}
By $\frac{1}{2}\pi<\alpha<\frac{3}{4}\pi$, we get two solutions
\[
(a,\cos\alpha,\sin\alpha)
=(1,-\tfrac{1}{2},\tfrac{\sqrt{3}}{2}),\;
(12\zeta^2-6\zeta-3,\zeta,\sqrt{1-\zeta^2}).
\]
where $\zeta=-0.41964$ is the only real number satisfying $4\zeta^3-4\zeta^2+1=0$, and $a=1.63107$. 

The first solution implies $\alpha=\frac{2}{3}\pi$. Substituting into $H_{9.1}$, we find all edges have length $1$ and all angles have value $\frac{2}{3}\pi$. Therefore the hexagon is regular. 

The second solution gives a hexagon that has two ${\mc D}_1$-tilings $243\tilde{3}\tilde{3}\tilde{5}$ and $24\tilde{3}\tilde{4}5\tilde{5}$, in Figure \ref{H9A}. However, if the ${\mc D}_1$-tilings are extended to ${\mc D}_2$-tilings, then the hexagon should fill the third angle at the degree $3$ vertex $\bullet$. This implies that the hexagon has an angle between two $a$, a contradiction.

\begin{figure}[htp]
\centering
\begin{tikzpicture}[>=latex,scale=1]


\begin{scope}[yshift=-0.6cm, scale=1.2]

\draw
	(0,0) -- (1.6313,0) -- ++(80.745:1) -- ++(145.931:1) -- ++(164.441:1) -- ++(245.186:1) -- ++(294.8137:1);

\node at (0.7,-0.136) {\tiny $1.631$};
\node at (1.28,0.14) {\tiny $99.248^{\circ}$};
\node at (0.42,0.14) {\tiny $114.812^{\circ}$};
\node at (1.33,0.9) {\tiny $114.812^{\circ}$};
\node at (0.1,0.92) {\tiny $130.376^{\circ}$};
\node[rotate=-60] at (0.24,1.4) {\tiny $99.248^{\circ}$};
\node[rotate=-25] at (0.9,1.4) {\tiny $161.503^{\circ}$};

\end{scope}

\begin{scope}[xshift=4.3cm, scale=0.5]

\foreach \a/\b/\c/\d in {1/0/0/0, -1/0/0/0, 1/114.8/3.44/-0.23, 1/80.745/2.76/1.25,    1/180/0.97/3.366, -1/130.4/-1.38/0.64, -1/180/-0.416/-0.91}
\draw[shift={(\c cm, \d cm)}, rotate=\b, yscale=\a]
	(0,0) -- (1.6313,0) -- ++(80.745:1) -- ++(145.931:1) -- ++(164.441:1) -- ++(245.186:1) -- ++(294.8137:1);

\fill (2.75,1.25) circle (0.1);

\node at (3.35,0.5) {\tiny $a$};
\node at (3.1,2) {\tiny $a$};
	
\end{scope}

\begin{scope}[xshift=8cm, scale=0.5]

\foreach \a/\b/\c/\d in {1/0/0/0, -1/0/0/0, 1/114.8/3.44/-0.23, 1/80.745/2.76/1.25,    -1/49.5/-0.41/2.73, -1/180/-0.415/0.916, 1/180/-0.415/0.916}
\draw[shift={(\c cm, \d cm)}, rotate=\b, yscale=\a]
	(0,0) -- (1.6313,0) -- ++(80.745:1) -- ++(145.931:1) -- ++(164.441:1) -- ++(245.186:1) -- ++(294.8137:1);

\fill (2.75,1.25) circle (0.1);

\node at (3.35,0.5) {\tiny $a$};
\node at (3.1,2) {\tiny $a$};

\end{scope}

\end{tikzpicture}
\caption{$H_{9.1}$, and two ${\mc D}_1$-tilings.}
\label{H9A}
\end{figure}
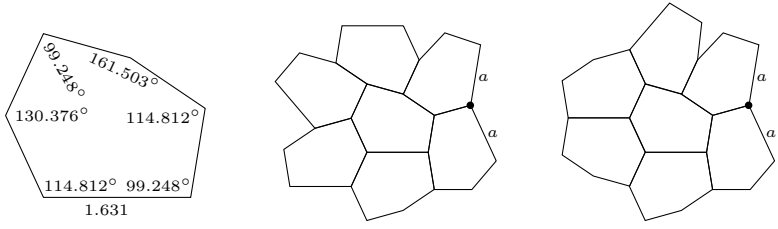

\medskip

\noindent{\em Hexagon $H_{9.2}$}

\medskip

Substituting the setting for the edges and angles into \eqref{hexeq}, we get
\[
a+e^{i(\pi-\alpha)}+e^{i(\frac{8}{3}\pi-3\alpha)}+e^{i(\frac{5}{3}\pi-\alpha)}+e^{i(-\frac{2}{3}\pi+3\alpha)}+e^{i(\frac{7}{3}\pi-\alpha)}=0.
\]
The imaginary part is actually identically $0$, and the real part becomes
\[
a =\cos \alpha
+2\cos (\tfrac{1}{3}\pi+3\alpha)
-\cos (\tfrac{1}{3}\pi+\alpha)
-\cos (\tfrac{1}{3}\pi-\alpha)
=2\cos (\tfrac{1}{3}\pi+3\alpha).
\]
By $\frac{1}{2}\pi<\alpha<\frac{5}{6}\pi$ and $a>0$, we get $\frac{1}{2}\pi<\alpha<\frac{13}{18}\pi$.

Figure \ref{H9B} shows the hexagon $H_{9.2}$ (with $\alpha=110^{\circ}$ and $
a=\sqrt{3}$) and the two ${\mc D}_1$-tilings $2\tilde{3}2\tilde{1}2\tilde{5}$ and $\tilde{2}1\tilde{0}\tilde{1}2\tilde{5}$. The only place that may cause overlapping in the ${\mc D}_1$-tilings is at $\circ$. The sum of the two angles at the vertex is $2(\frac{10}{3}\pi-4\alpha)$, and should be $<2\pi$. Combined with $\frac{1}{2}\pi<\alpha<\frac{13}{18}\pi$, we get the range $\frac{7}{12}\pi<\alpha<\frac{13}{18}\pi$ for the existence of the ${\mc D}_1$-tilings.

The complementary angle at $\bullet$ is $4\pi-5\alpha$. If the ${\mc D}_1$-tilings are extended to ${\mc D}_2$-tilings, then the angle equals one of the six angles of the hexagon. In all the cases, we get $\alpha=\frac{2}{3}\pi$, and the hexagon is regular.

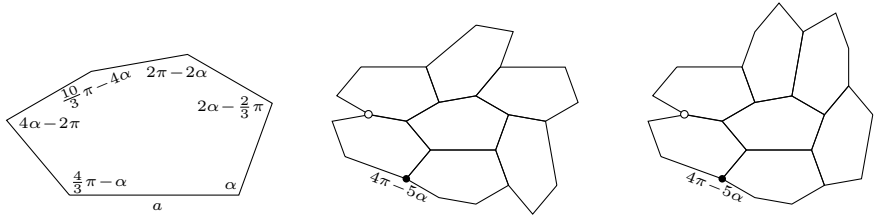
\begin{figure}[htp]
\centering
\begin{tikzpicture}[>=latex,scale=1]


\begin{scope}[yshift=-0.6cm, scale=1.3]

\draw
	(0,0) -- (1.732,0) -- ++(70:1) -- ++(150:1) -- ++(190:1) -- ++(210:1) -- ++(310:1);

\node at (0.9,-0.12) {\tiny $a$};
\node at (0.3,0.13) {\tiny $\frac{4}{3}\pi\!-\!\alpha$};
\node at (1.65,0.1) {\tiny $\alpha$};
\node at (1.65,0.9) {\tiny $2\alpha\!-\!\frac{2}{3}\pi$};
\node at (1.1,1.25) {\tiny $2\pi\!-\!2\alpha$};
\node[rotate=23] at (0.27,1.1) {\tiny $\frac{10}{3}\pi\!-\!4\alpha$};
\node at (-0.2,0.73) {\tiny $4\alpha\!-\!2\pi$};

\end{scope}

\begin{scope}[xshift=4.8cm, scale=0.5]

\foreach \a/\b/\c/\d in {1/0/0/0, -1/0/0/0, 1/-20/-2.272/-0.175, -1/0/-1.849/2.21,   1/220/1.214/3.323, -1/0/1.85/2.205, 1/100/3.36/-0.94}
\draw[shift={(\c cm, \d cm)}, rotate=\b, yscale=\a]
	(0,0) -- (1.732,0) -- ++(70:1) -- ++(150:1) -- ++(190:1) -- ++(210:1) -- ++(310:1);

\fill (-130:1) circle (0.1);

\filldraw[fill=white] (-1.65,0.95) circle (0.1);

\node[rotate=-22] at (-130:1.3) {\tiny $4\pi\!-\!5\alpha$};
	
\end{scope}

\begin{scope}[xshift=9cm, scale=0.5]

\foreach \a/\b/\c/\d in {1/0/0/0, -1/0/0/0, 1/-20/-2.272/-0.175, -1/0/-1.849/2.21,  -1/-100/1.51/3.144, 1/-100/1.51/3.143, -1/-100/3.36/0.94}
\draw[shift={(\c cm, \d cm)}, rotate=\b, yscale=\a]
	(0,0) -- (1.732,0) -- ++(70:1) -- ++(150:1) -- ++(190:1) -- ++(210:1) -- ++(310:1);

\fill (-130:1) circle (0.1);

\filldraw[fill=white] (-1.65,0.95) circle (0.1);

\node[rotate=-22] at (-130:1.3) {\tiny $4\pi\!-\!5\alpha$};

\end{scope}

\end{tikzpicture}
\caption{$H_{9.2}$, and two ${\mc D}_1$-tilings.}
\label{H9B}
\end{figure}

\medskip

\noindent{\em Hexagon $H_{9.3}$}

\medskip

Substituting the setting for the edges and angles into \eqref{hexeq}, we get
\[
a+e^{i(\pi-\alpha)}+e^{i(2\pi-2\alpha)}+e^{i\pi}+e^{i(4\pi-4\alpha)}+e^{i(-\pi+4\alpha)}=0.
\]
Taking the imaginary part, we get
\[
0=\sin\alpha-\sin 2\alpha-2\sin 4\alpha
=-(2\cos\alpha+1)(8\cos^2\alpha-4\cos\alpha-1)\sin\alpha.
\]
By $\frac{1}{2}\pi<\alpha<\frac{3}{4}\pi$, we get two solutions
\[
(a,\cos\alpha,\sin\alpha) 
= (1,-\tfrac{1}{2},\tfrac{\sqrt{3}}{2}),\;
(\tfrac{7}{4},\tfrac{1-\sqrt{3}}{4},\tfrac{\sqrt{12+2\sqrt{3}}}{4}).
\]
Similar to $H_{9.1}$, the first solution implies the regular hexagon. The second solution gives a hexagon that has three ${\mc D}_1$-tilings $22\tilde{1}\tilde{2}\tilde{3}\tilde{5}$, $2\tilde{2}\tilde{2}4\tilde{3}\tilde{5}$, $2\tilde{2}\tilde{2}\tilde{4}5\tilde{5}$, in Figure \ref{H9C}. Unfortunately, the tiles in the three ${\mc D}_1$-tilings overlap at ?. Therefore $H_{9.3}$ actually does not have ${\mc D}_1$-tiling.

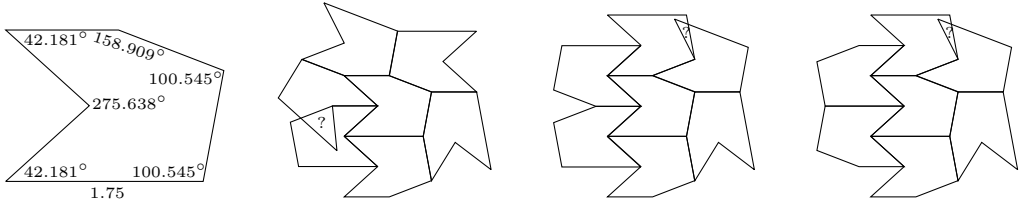
\begin{figure}[htp]
\centering
\begin{tikzpicture}[>=latex]

\begin{scope}[shift={(-4.5cm,-0.6cm)}, scale=1.5]

\draw
	(0,0) -- (1.75,0) -- ++(79.4547:1) -- ++(158.9094:1) -- ++(-1,0) -- ++(317.8188:1) -- ++(222.1817:1);

\node at (0.9,-0.1) {\tiny $1.75$};
\node at (1.45,0.1) {\tiny $100.545^{\circ}$};
\node at (0.45,0.1) {\tiny $42.181^{\circ}$};
\node at (0.45,1.27) {\tiny $42.181^{\circ}$};
\node at (1.1,0.7) {\tiny $275.638^{\circ}$};
\node at (1.6,0.92) {\tiny $100.545^{\circ}$};
\node[rotate=-18] at (1.1,1.2) {\tiny $158.909^{\circ}$};

\end{scope}

\begin{scope}[scale=0.6]

\foreach \a/\b/\c/\d in {1/0/0/0, -1/0/0/0, -1/-21/-0.45/2.96, -1/180/0.74/-0.67, -1/138/-0.16/-0.32, 1/100.545/3.255/-0.74, 1/180/2.93/2.33}
\draw[shift={(\c cm, \d cm)}, rotate=\b, yscale=\a]
	(0,0) -- (1.75,0) -- ++(79.4547:1) -- ++(158.9094:1) -- ++(-1,0) -- ++(317.8188:1) -- ++(222.1817:1);

\node at (-0.5,0.3) {\tiny ?};

\end{scope}

\begin{scope}[xshift=3.5cm, scale=0.6]

\foreach \a/\b/\c/\d in {1/0/0/0, -1/0/0/0, -1/0/0/2.69, 1/180/0.74/2.01, -1/180/0.74/-0.68, 1/100.545/3.255/-0.74, -1/-21/1.486/2.6}
\draw[shift={(\c cm, \d cm)}, rotate=\b, yscale=\a]
	(0,0) -- (1.75,0) -- ++(79.4547:1) -- ++(158.9094:1) -- ++(-1,0) -- ++(317.8188:1) -- ++(222.1817:1);

\node at (1.75,2.3) {\tiny ?};

\end{scope}

\begin{scope}[xshift=7cm, scale=0.6]

\foreach \a/\b/\c/\d in {1/0/0/0, -1/0/0/0, -1/0/0/2.69, 1/180/0.74/0.67, -1/180/0.74/0.67, 1/100.545/3.255/-0.74, -1/-21/1.486/2.6}
\draw[shift={(\c cm, \d cm)}, rotate=\b, yscale=\a]
	(0,0) -- (1.75,0) -- ++(79.4547:1) -- ++(158.9094:1) -- ++(-1,0) -- ++(317.8188:1) -- ++(222.1817:1);

\node at (1.75,2.3) {\tiny ?};

\end{scope}

\end{tikzpicture}
\caption{$H_{9.3}$, and three ${\mc D}_1$-tilings.}
\label{H9C}
\end{figure}

\medskip

\noindent{\em Hexagons $H_{9.4}$ through $H_{9.14}$}

\medskip

First we use the method in $H_{9.i}$, $i=1,2,3$, to find $\alpha$ within the range. The angle $\alpha=\frac{2}{3}\pi$ is always a solution, and corresponds to the regular hexagon. The only other solutions are the following, and we further calculate $a$:
\begin{itemize}
\item $H_{9.9}$: $(\alpha,a)=(\pi-\arctan\frac{1}{\sqrt{11}},-\frac{\sqrt{3}}{\sqrt{11}})$. 
\item $H_{9.12}$: $(\alpha,a)=(\arccos \zeta_1,-0.53740),(\arccos \zeta_2,-3.93543)$, where $\zeta_1<\zeta_2<\zeta_3$ are the three roots of $8\zeta^3-8\zeta-1=0$. 
\end{itemize}
Since $a<0$ in all the cases, the solutions are dismissed. 

\subsection{Rank 10}

We carry out the same process for the rank 10 tilings, removing those that can imply Reinhardt hexagons (including $H_{9.i}$, $i=3,4,\dots,14$). What remain are seven equilateral hexagons
\[
|\bar{0}|=|\bar{1}|=|\bar{2}|=|\bar{3}|=|\bar{4}|=|\bar{5}|.
\]
We omit the above equality in the following list.
\begin{itemize}
\item $H_{10.1}$:
$052\tilde{1}2\tilde{5}$,
$\tilde{0}\tilde{3}2\tilde{1}2\tilde{5}$.

$[5]=\alpha$,
$[0]=\frac{14}{3}\pi-6\alpha$,
$[1]=4\alpha-2\pi$,
$[2]=3\alpha-\frac{4}{3}\pi$,
$[3]=2\pi-2\alpha$,
$[4]=\frac{2}{3}\pi$.
$\frac{1}{2}\pi<\alpha<\frac{7}{9}\pi$.
\item $H_{10.2}$:
$05\tilde{4}52\tilde{5}$,
$05\tilde{4}5\tilde{2}\tilde{5}$.

$[5]=\alpha$,
$[0]=\frac{8}{3}\pi-3\alpha$,
$[1]=[3]=2\pi-2\alpha$,
$[2]= 6\alpha-\frac{10}{3}\pi$,
$[4]=\frac{2}{3}\pi$.
$\frac{5}{9}\pi<\alpha<\frac{8}{9}\pi$.
\item $H_{10.3}$: 
$010503$,
$0105\tilde{0}\tilde{1}$,
$010\tilde{5}\tilde{4}3$,
$\tilde{4}10\tilde{5}45$,
$2\tilde{5}0\tilde{5}4\tilde{5}$,
$\tilde{4}10\tilde{5}\tilde{4}\tilde{3}$,
$\tilde{0}\tilde{5}0\tilde{5}\tilde{4}3$.

$[0]=\alpha$,
$[1]=\frac{10}{3}\pi-4\alpha$,
$[2]=4\alpha-2\pi$,
$[3]=\frac{4}{3}\pi-\alpha$,
$[4]=2\pi-2\alpha$,
$[5]=2\alpha-\frac{2}{3}\pi$. 
$\frac{1}{2}\pi<\alpha<\frac{5}{6}\pi$.
\item $H_{10.4}$:
$\tilde{0}0\tilde{5}0\tilde{4}\tilde{2}$.

$[0]=\alpha$,
$[1]=\frac{10}{3}\pi-4\alpha$,
$[2]=4\alpha-2\pi$,
$[3]=\frac{4}{3}\pi-\alpha$,
$[4]=2\alpha-\frac{2}{3}\pi$,
$[5]=2\pi-2\alpha$.
$\frac{1}{2}\pi<\alpha<\frac{5}{6}\pi$.
\item $H_{10.5}$:
$0\tilde{5}03\tilde{2}3$.

$[1]=\alpha$,
$[0]=\frac{10}{3}\pi-4\alpha$,
$[2]=2\alpha-\frac{2}{3}\pi$,
$[3]=4\alpha-2\pi$,
$[4]=\frac{4}{3}\pi-\alpha$,
$[5]=2\pi-2\alpha$.
$\frac{1}{2}\pi<\alpha<\frac{5}{6}\pi$.
\item $H_{10.6}$: 
$41\tilde{0}1\tilde{0}1$,
$01\tilde{0}1\tilde{0}\tilde{1}$,
$\tilde{4}1\tilde{0}1\tilde{0}1$,
$\tilde{0}\tilde{5}\tilde{0}1\tilde{0}\tilde{1}$.

$[3]=\alpha$,
$[0]=[2]=2\alpha-\frac{2}{3}\pi$,
$[1]=[5]=2\pi-2\alpha$,
$[4]=\frac{4}{3}\pi-\alpha$.
$\frac{1}{3}\pi<\alpha<\pi$.
\item $H_{10.7}$:
$\tilde{0}21\tilde{0}10$.

$[3]=\alpha$,
$[0]=[2]=2\pi-2\alpha$,
$[1]=[5]=2\alpha-\frac{2}{3}\pi$,
$[4]=\frac{4}{3}\pi-\alpha$.
$\frac{1}{3}\pi<\alpha<\pi$.
\end{itemize}

We use \eqref{hexeq} to calculate $\alpha$. For all the cases, the only solution of $\alpha$ within the range is $\alpha=\frac{2}{3}\pi$. This means the hexagon is regular. 

The calculation for $H_{10.1}$ is the most complicated. We  illustrate the calculation using this case. The equation \eqref{hexeq} for the hexagon is
\[
1
+e^{i(\pi-\alpha)}
+e^{i(-\frac{8}{3}\pi+5\alpha)}
+e^{i(\frac{1}{3}\pi+\alpha)}
+e^{i(\frac{8}{3}\pi-2\alpha)}
+e^{i\frac{5}{3}\pi}=0.
\]
In terms of $\theta=\alpha-\frac{2}{3}\pi$ and $z=e^{i\theta}$, the equation is the same as
\begin{align*}
0&=z^7+e^{i\frac{1}{3}\pi}z^3+(e^{-i\frac{2}{3}\pi}-1)z^2+e^{-i\frac{1}{3}\pi}z+e^{-i\frac{2}{3}\pi} \\
&=(z-1)(z-e^{i\frac{4}{3}\pi})(z^5+e^{-i\frac{1}{3}\pi}z^4+z^2+z+e^{i\frac{4}{3}\pi}).
\end{align*}
Since $\frac{1}{2}\pi<\alpha<\frac{7}{9}\pi$, we get $-\frac{1}{6}\pi<\theta<\frac{1}{9}\pi$. The root of the factor $z-1$ is $\theta=0$, which corresponds to the regular hexagon. The root of the factor $z-e^{i\frac{4}{3}\pi}$ is $\theta=\frac{4}{3}\pi$, which is outside the range. For the roots of the quintic factor, we consider
\[
z^{-3}(z^5+e^{-i\frac{1}{3}\pi}z^4+z^2+z+e^{i\frac{4}{3}\pi})
=z^2+z^{-2}+e^{i(-\frac{1}{3}\pi+\theta)}+e^{-i\theta}+e^{i(\frac{4}{3}\pi-3\theta)}.
\]
For $z=e^{i\theta}$, the imaginary part vanishes
\[
0
=\sin(-\tfrac{1}{3}\pi+\theta) 
-\sin\theta
+\sin(\tfrac{4}{3}\pi-3\theta)
=-4\cos 2\theta\sin(\tfrac{1}{3}\pi-\theta). 
\]
The equation has no solution for $-\frac{1}{6}\pi<\theta<\frac{1}{9}\pi$.

\section{Comments}
\label{comment}

The conclusion of Section \ref{proof} may be summarised as the following.

\begin{theorem}\label{d1}
There are seven possibly non-Reinhardt hexagons that have ${\mc D}_1$-tilings. The total number of such ${\mc D}_1$-tilings is 33.
\end{theorem}

The seven hexagons are $H_{7.1}$, $H_{8.1}$, $H_{8.2}$, $H_{8.3}$, $H_{8.4}$, $H_{9.1}$, $H_{9.2}$. In Section \ref{proof}, we actually proved that, if any of the 33 ${\mc D}_1$-tilings can be extended to ${\mc D}_2$-tiling, then the corresponding hexagon becomes regular. The proof uses the geometrical implication of the ${\mc D}_2$-tilings on the boundary configurations of ${\mc D}_1$-tilings. 

There is a more combinatorial proof. The extension of a ${\mc D}_1$-tiling to a ${\mc D}_2$-tiling is the same as picking one ${\mc D}_1$-tiling around each of the six non-central tiles (labeled 2 through 7 in Figure \ref{geom3}) in the ${\mc D}_1$-tiling, such that six ${\mc D}_1$-tilings are compatible. For the proof of the main theorem, we may restrict all these ${\mc D}_1$-tilings to be among the 33 (and their circular permutations and flippings). This can be carried out by computer program. We find that it is impossible to have a ${\mc D}_2$-tiling, such that all seven ${\mc D}_1$-tilings in the ${\mc D}_2$-tiling are among the 33. In other words, in any ${\mc D}_2$-tiling, at least one of the seven ${\mc D}_1$-tilings is not among the 33. By Theorem \ref{d1}, this means that this ${\mc D}_1$-tiling implies the hexagon is a Reinhardt hexagon.

In \cite{heesch}, Heesch found a pentagon that has a one layer tiling but not a two layer tiling. Specifically, we have a more general notion of ${\mc D}_k$-tiling, by inductively define general ${\mc D}_{k+1}$-tiling to consist of all tiles touching the general ${\mc D}_k$-tiling. This includes all the tiles in the general ${\mc D}_k$-tiling, and all tiles sharing common edges with the general ${\mc D}_k$-tiling, and even the tiles sharing common points with the general ${\mc D}_k$-tiling. Moreover, the tilings are not required to be edge-to-edge, and there is no three tile requirement. A polygon (or more generally, a topological disk) has {\em Heesch number} $k$ if it has general ${\mc D}_k$-tiling but not general ${\mc D}_{k+1}$-tiling. 

A polygon that can tile the plane has Heesch number $\infty$. The example in \cite{heesch} has Heesch number 1. The largest known Heesch number is 6 \cite{basic}.

The seven hexagons $H_{7.1}$, $H_{8.1}$, $H_{8.2}$, $H_{8.3}$, $H_{8.4}$, $H_{9.1}$, $H_{9.2}$ obtained in this paper have ${\mc D}_1$-tilings but not ${\mc D}_2$-tilings, as long as they are not Reinhardt hexagons. However, the ${\mc D}_k$-tilings are restricted to edge-to-edge, and all vertices haing degree $3$. We may regard the hexagons to have {\em special} Heesch number 1. The general problem of the maximal Heesch number for hexagons remains open.

\end{document}